\documentclass[a4paper,12pt,oneside]{amsart} 
\usepackage[left=3cm, right=3cm,top=3.5cm,bottom=3.5cm]{geometry}
\usepackage{algorithm, booktabs}
\usepackage{cite}
 \usepackage{amsaddr}

\theoremstyle{definition}

\newcommand{\beqa}{\begin{eqnarray*}}
\newcommand{\eeqa}{\end{eqnarray*}}

\DeclareMathOperator*{\supp}{supp}

\def\<{\left<}
\def\>{\right>}

\def\mv1{M_v^1}

\usepackage{upgreek}
\usepackage{mathtools}
\usepackage{lscape}
\usepackage[justification=raggedright]{caption}
\usepackage{titlesec}
\titleformat{\section}
  {\centering \normalfont\fontsize{15}{15}\bfseries}{\thesection}{1em}{}
  \titleformat{\subsection} 
  {\normalfont\fontsize{13}{13}\bfseries}{\thesubsection}{1em}{}
\usepackage{enumitem}
\usepackage{amsfonts}
\usepackage{amsmath}
\usepackage{amssymb}
\usepackage{hyperref}
\usepackage{faktor}
\usepackage{amscd}
\usepackage{array}
\usepackage{subfigure}
\usepackage{mathdots}
\usepackage{mathrsfs}
\usepackage{graphicx}
\usepackage{color}
\usepackage{pgf}
\usepackage{multirow}
\usepackage{tikz}
\usepackage{graphicx}

\newcommand{\N}{\ensuremath{\mathbb{N}}}

\newcommand{\sd}{\mathbb{S}^d}

\renewcommand{\S}{\ensuremath{\mathbb{S}}}

\newcommand{\R}{\ensuremath{\mathbb{R}}}

\def\3{\ss}

\newmuskip\pFqmuskip
\newcommand*\pFq[6][8]{
  \begingroup 
  \pFqmuskip=#1mu\relax
  \begingroup\lccode`\~=`\,
  \lowercase{\endgroup\let~}\pFqcomma
  {}_{#2}F_{#3}{\left(\genfrac..{0pt}{}{#4}{#5};#6\right)}%
  \endgroup
}
\newcommand*\pRegFq[6][8]{
  \begingroup 
  \pFqmuskip=#1mu\relax
  \begingroup\lccode`\~=`\,
  \lowercase{\endgroup\let~}\pFqcomma
  {}_{#2}\tilde{F}_{#3}{\left(\genfrac..{0pt}{}{#4}{#5};#6\right)}%
  \endgroup
}
\newcommand{\pFqcomma}{\mskip\pFqmuskip}

\DeclareMathOperator{\invariant}{inv}

\DeclareMathOperator*{\spann}{span}
\DeclareMathOperator{\dist}{dist}

\DeclareMathOperator{\Pol}{Pol}
\DeclareMathOperator{\Poltd}{Pol_{\leq t}}
\DeclareMathOperator{\PoltdG}{Pol^G_{\leq t}}
\DeclareMathOperator{\Polsd}{Pol_{\leq s}}
\DeclareMathOperator{\PolsdG}{Pol^G_{\leq s}}

\usepackage{etoolbox}
\makeatletter
\AtBeginEnvironment{align*}{\let\SetColor\@gobble \color{display}%
                            \Patchintertext}
\AtBeginEnvironment{align}{\let\SetColor\@gobble \color{display}\Patchintertext}
\AtBeginEnvironment{equation*}{\let\SetColor\@gobble \color{display}}
\AtBeginEnvironment{equation}{\let\SetColor\@gobble \color{display}}

\def\Patchintertext{\let\oldintertext@\intertext@
                    \def\intertext@{\oldintertext@\Patchintertext@}}
\def\Patchintertext@ {\let\oldintertext\intertext
                      \def\intertext ##1{\oldintertext
                      {\color{black}\let\SetColor\color ##1}}}
\makeatother
\everymath\expandafter{\the\everymath\SetColor{inline}}
\everydisplay\expandafter{\the\everydisplay\SetColor{display}}
\definecolor{display}{rgb}{0,0,.5}
\definecolor{inline}{rgb}{0,0,.3}
\let\SetColor\color

\theoremstyle{theorem}
\newtheorem{thm}{Theorem}[section]
\newtheorem{lemma}[thm]{Lemma}
\newtheorem{corollary}[thm]{Corollary}
\newtheorem{proposition}[thm]{Proposition}

\theoremstyle{definition}
\newtheorem{definition}[thm]{Definition}
\newtheorem{remark}[thm]{Remark}
\newtheorem{example}[thm]{Example}

\usepackage{tcolorbox}
\tcbset{colback=gray!1,colframe=gray!20!white}

\begin{document}
\title[]{
Hybrid spherical designs}
\date{Martin \today}
\author[M.~Ehler]{Martin Ehler}
\address[M.~Ehler]{Faculty of Mathematics, University of Vienna,
Austria
}
\email{martin.ehler@univie.ac.at}

\begin{abstract} 
Spherical $t$-designs are finite point sets on the unit sphere that enable exact integration of polynomials of degree at most $t$ via equal-weight quadrature. This concept has recently been extended to spherical $t$-design curves by the use of normalized path integrals. However, explicit examples of such curves are rare.  

We construct new spherical $t$-design curves for small $t$ based on edge-transitive convex polytopes. We then introduce hybrid $t$-designs that combine points and curves for exact polynomial integration of higher degree. Our constructions use the edges and vertices of dual pairs of convex polytopes and polynomial invariants of their symmetry group. A notable result is a hybrid $t$-design for $t=19$.
\end{abstract}

\maketitle

\section{Introduction}
Integration on the unit sphere serves as a fundamental tool in mathematics, physics, and engineering for analyzing and modeling phenomena with spherical symmetry. An effective approach for such integration is the concept of spherical designs. These are sets of points on the sphere that enable exact integration of polynomials up to a specific degree by using equal-weight quadrature formulas. Quadrature in general is widely studied in numerical analysis, approximation theory, and signal processing, and although equal weights are not always necessary, they do simplify the formulas and make them more convenient and easier to use for practitioners.

\subsubsection*{Points:} Formally, for \( t \in \mathbb{N} \), let $\Poltd$ denote the space of all polynomials of total degree at most \( t \) in \( d+1 \) variables, restricted to the sphere $\mathbb{S}^d = \{x \in \mathbb{R}^{d+1} : \|x\| = 1\}$.  
A spherical $t$-design is a finite set of points \( X\subseteq \S^d\) with cardinality \( |X| \), such that the normalized surface integral on the sphere can be exactly determined by the equal-weight quadrature formula  
\begin{equation}\label{eq:design points def}
\frac{1}{|X|} \sum_{x \in X} p(x) = \int_{\mathbb{S}^d} p \,, \quad \text{for all } p \in \Poltd.
\end{equation}
This concept dates back to work by McLaren in 1963 \cite{McLaren63}, although the term `spherical design' was coined later by Delsarte, Goethals, and Seidel in the 1970s \cite{Delsarte:1977aa}. Since then, spherical $t$-designs have inspired a rich body of research investigating their explicit construction through analytic, algebraic, and numerical methods \cite{Bajnok91,Korevaar93, Bondarenko:2011kx, Bondarenko:2015eu, Womersley:2018we, Danev2001, Harpe:2004, Hong82, Reznick:1995, Sloane:2003zp}.  
\subsubsection*{Curves:}
While spherical designs have been extensively studied in the discrete setting, their continuous counterparts, spherical design curves, have been introduced only recently in \cite{EG:2023}. These closed curves define a continuous path on the sphere and achieve the same exact integration property via normalized path integrals, i.e., a continuous function $\gamma:[0,L]\rightarrow \S^d$ with $\gamma(0)=\gamma(L)$ is a $t$-design curve if
\begin{equation*}
\frac{1}{\ell(\gamma)}  \int_{\gamma} p = \int_{\sd} p\,,\quad\text{for all } p\in\Poltd\,,
\end{equation*}
where $\ell(\gamma)$ denotes the arc-length of $\gamma$. 

Although constructions of families of $t$-design curves are provided in \cite{EGK:24,Lindblad1}, explicit examples remain rare, even for modest values of $t$. Some new $t$-design curves are shown in Figures \ref{fig:summary00000}, \ref{fig:summary0000}, and  \ref{fig:summary000}.

\subsubsection*{Curves for mobile sampling:}
The concept of spherical design curves in \cite{EG:2023} is motivated by mobile sampling, which studies the reconstruction of a function from its values along a curve \cite{Unnikrishnan:2013df,Unnikrishnan:2013aa,Grochenig:2015ya}. Unlike discrete point sampling, such as with spherical design points that require multiple sensors, mobile sampling utilizes a single sensor that moves along a curved trajectory. An illustrative case of mobile sampling on the sphere involves gathering environmental data along an airplane flight path or the trajectory of a satellite, with the objective of reconstructing global information over the entire globe \( \mathbb{S}^2 \). 

From a practical perspective, it makes sense to combine both sampling approaches by using a mobile device alongside static sensors where available.  
However, point-based and curve-based methods have traditionally been studied separately, and their combined use has remained largely unexplored within sampling theory and designs.   

\subsubsection*{Hybrid designs:}
We propose the new notion of hybrid designs, where we consider both, a curve $\gamma:[0,L]\rightarrow\sd$ and finitely many points $X\subseteq \sd$. We call the pair
$(\gamma,X)$ a hybrid $t$-design if there is a balancing constant $0\leq \beta\leq 1$ such that 
\begin{equation*}
\frac{\beta}{|X|} \sum_{x\in X} p(x) +\frac{1-\beta}{\ell(\gamma)}  \int_{\gamma} p = \int_{\sd} p\,,\quad\text{for all } p\in\Poltd\,.
\end{equation*}
The benefit of hybrid designs in comparison to a `pure' curve is increase in strength. Given a $t$-design curve $\gamma$, we may add a set of points \( X \) to enhance the integration strength to \( \Pol_{\leq s} \) for some \( s > t \), see Figures \ref{fig:summary0}, \ref{fig:summary}, and \ref{fig:summary2}.

\subsection*{Main contributions}
In this work, we expand the catalog of spherical designs by constructing new \( t \)-design curves and hybrid designs. 
The key findings are summarized as follows:
\begin{itemize}
\item [(A)] 
We construct several new $t$-design curves, some are shown in Figures \ref{fig:summary00000}, \ref{fig:summary0000}, and  \ref{fig:summary000}. These closed curves are obtained by projecting the edges of distinct convex polytopes onto the sphere as spherical arcs, and those arcs form a curve on the sphere. This construction relies on the symmetry properties of edge-transitive polytopes, whose symmetry group by definition acts transitively on the edges. 

\begin{figure}
\subfigure[great circle]{
\includegraphics[width=.22\textwidth]{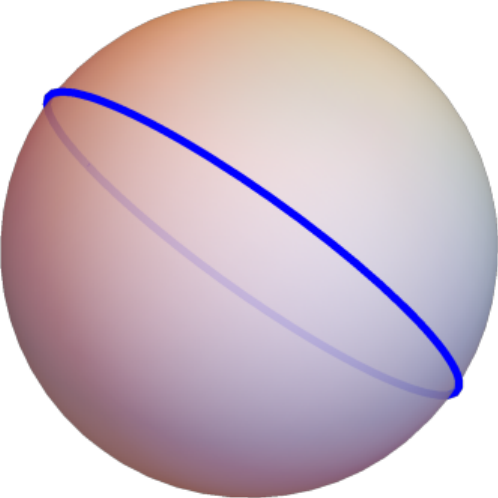}
}
\subfigure[tetrahedron]{
\includegraphics[width=.22\textwidth]{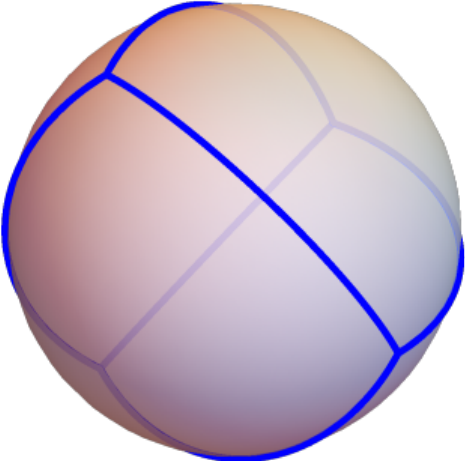}
}
\vspace{-.3cm}
\caption{$1$- and $2$-designs.}\label{fig:summary00000}
\end{figure}
\begin{figure}
\subfigure[octahedron]{
\includegraphics[width=.22\textwidth]{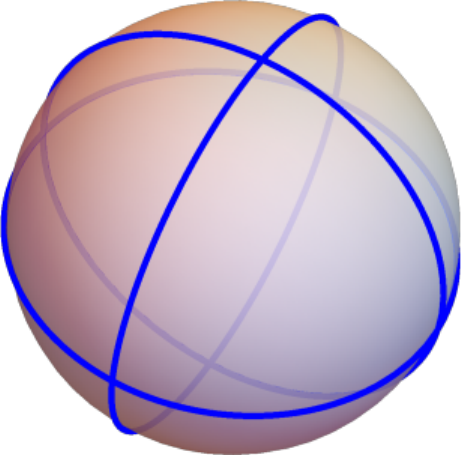}\label{fig:octa standard}
}
\subfigure[cube]{
\includegraphics[width=.22\textwidth]{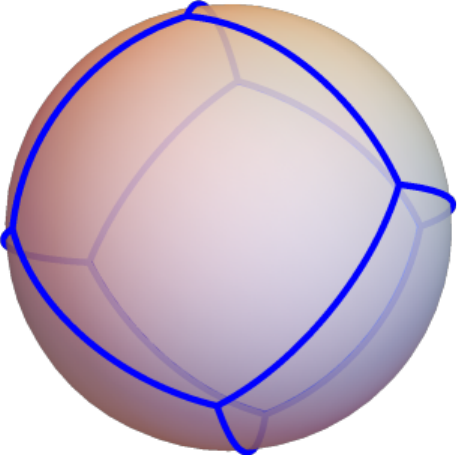}
}
\subfigure[rhombic \newline dodecahedron]{
\includegraphics[width=.22\textwidth]{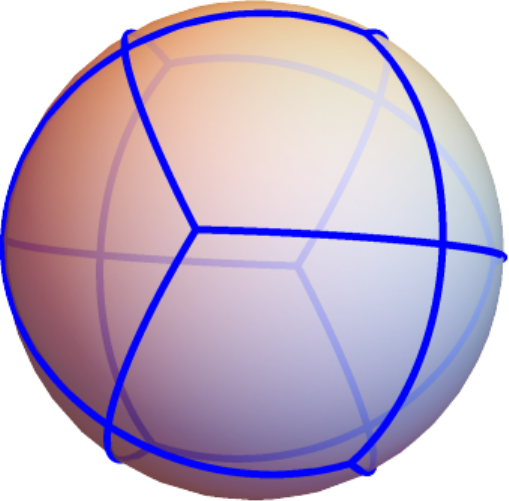}
}
\subfigure[cuboctahedron]{
\includegraphics[width=.22\textwidth]{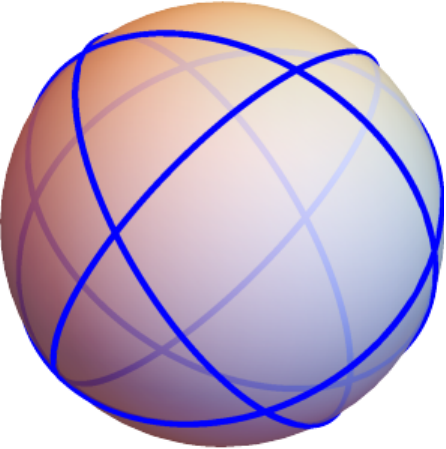}\label{fig:cubo intro}
}
\vspace{-.3cm}
\caption{$3$-designs.
}\label{fig:summary0000}
\end{figure}
\begin{figure}
\subfigure[dodecahedron]{
\includegraphics[width=.22\textwidth]{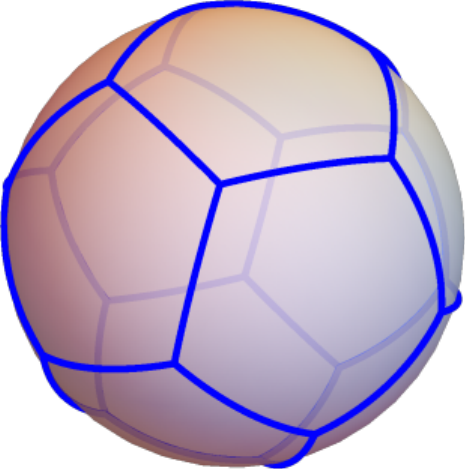}
}
\subfigure[icosahedron]{
\includegraphics[width=.22\textwidth]{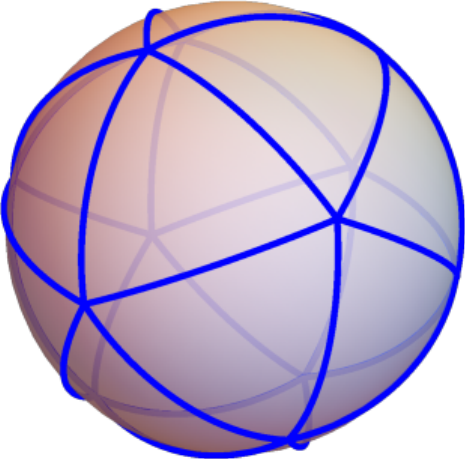}
}
\subfigure[rhombic \newline triacontahedron]{
\includegraphics[width=.22\textwidth]{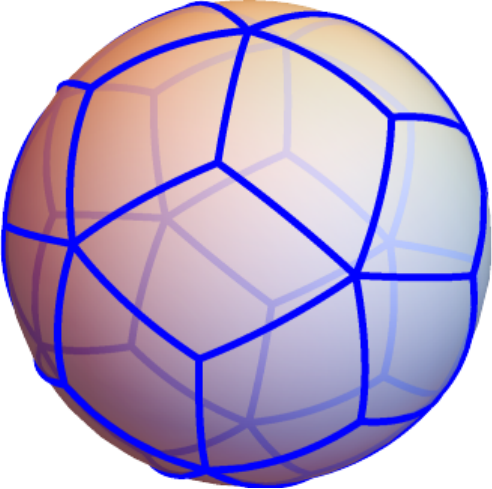}
}
\subfigure[icosidodecahedron]{
\includegraphics[width=.22\textwidth]{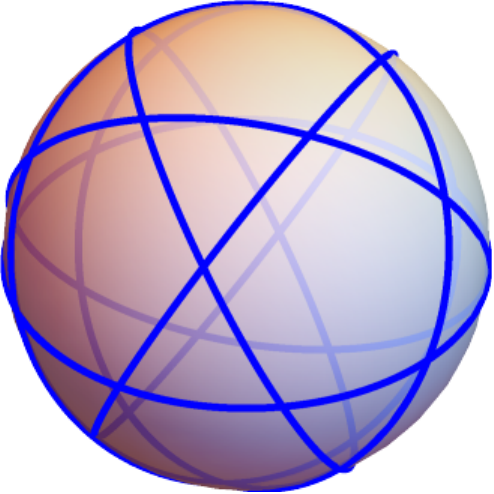}
\label{ggg0}}
\vspace{-.3cm}
\caption{$5$-designs.
}\label{fig:summary000}
\end{figure}

\item[(B)] We derive the first hybrid $t$-designs, some are shown in Figures \ref{fig:summary0}, \ref{fig:summary}, and \ref{fig:summary2}. 
Our primary construction principle relies on distinct pairs of dual polytopes and their symmetry group, so that the curve $\gamma$ is induced by the edges of the primal polytope projected onto the sphere and the points $X$ are the vertices of the dual polytope. 
We use invariant theory \cite{Meyer54} to determine the strength \( t \) of the hybrid design for several pairs of dual polytopes. This approach was guided by the ideas in \cite{McLaren63,Goethals:81b}, where multiple orbits are used to derive weighted integration formulas.

\begin{figure}
\subfigure[$2$-design, $\beta=\frac{1}{4}$]{
\includegraphics[width=.22\textwidth]{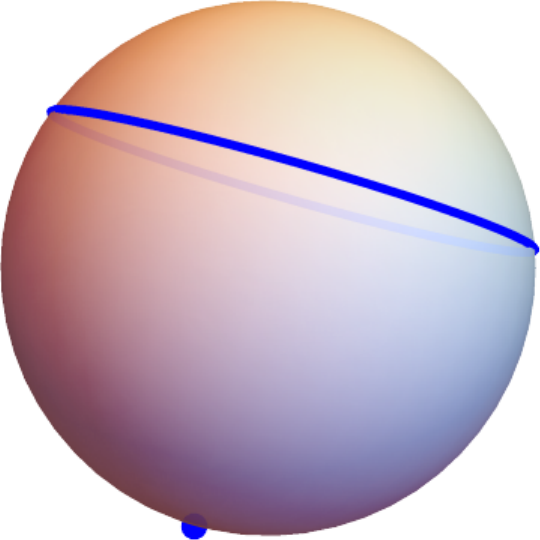}
\label{a}}
\subfigure[$2$-design, $\beta\approx \frac{1}{4}$]{
\includegraphics[width=.22\textwidth]{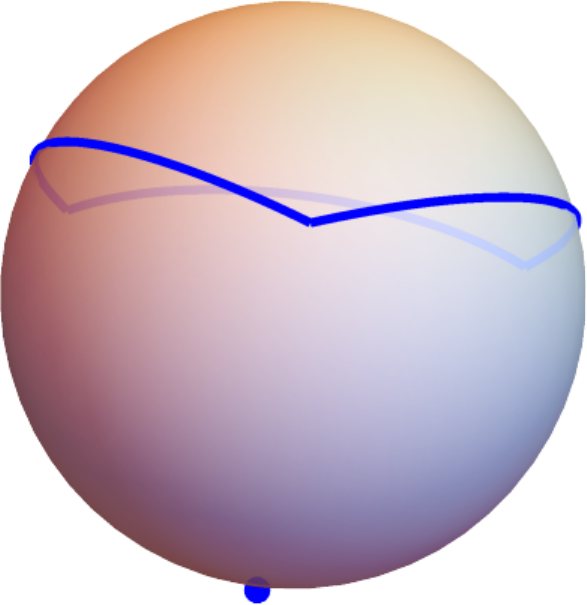}
\label{mixed:b}}
\subfigure[$3$-design, $\beta=\frac{1}{3}$]{
\includegraphics[width=.22\textwidth]{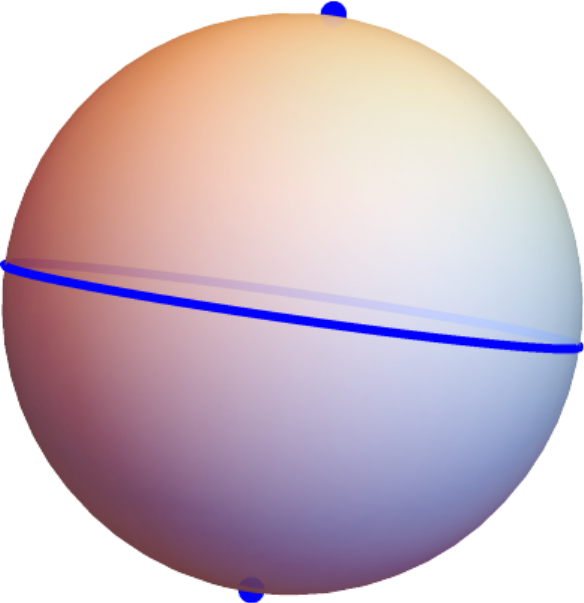}
\label{mixed:c}}
\subfigure[$3$-design, $\beta\approx \frac{1}{3}$]{
\includegraphics[width=.22\textwidth]{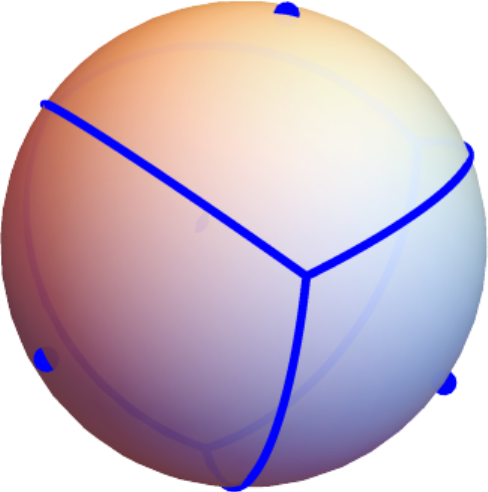}
\label{c}}
\vspace{-.3cm}
\caption{Hybrid $2$- and $3$-designs.
}\label{fig:summary0}
\end{figure}
\begin{figure}
\begin{minipage}[t]{\textwidth}
\subfigure[octahedron/cube]{
\includegraphics[width=.22\textwidth]{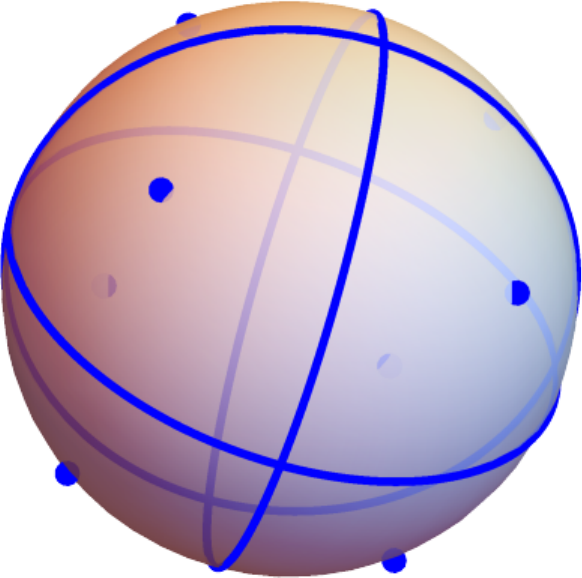}
\label{e}\hspace{-.1cm}}
\subfigure[cube/octahedron]{
\includegraphics[width=.22\textwidth]{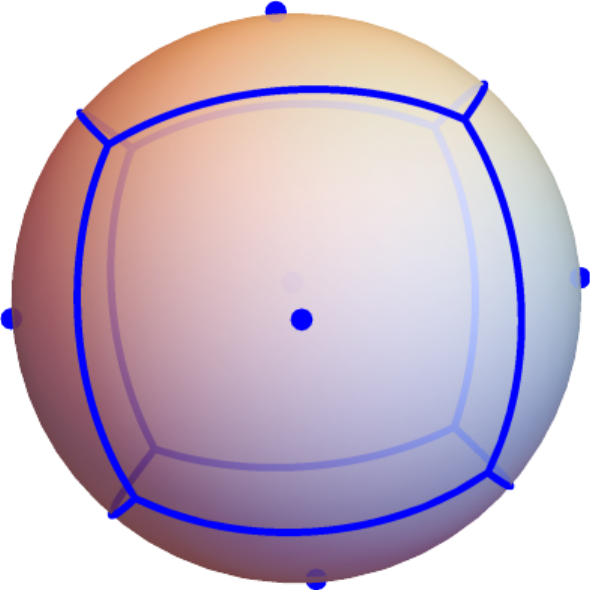}
\label{f}\hspace{-.1cm}}
\subfigure[rhombic dodecahedron/cuboctahedron]{
\includegraphics[width=.22\textwidth]{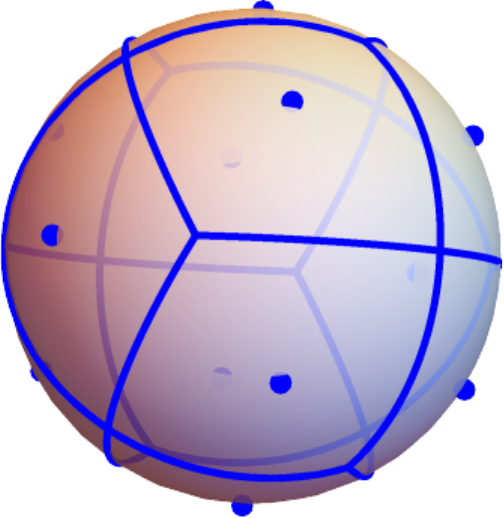}
\label{ggg}\hspace{-.1cm}}
\subfigure[cuboctahedron/rhombic dodecahedron]{
\includegraphics[width=.22\textwidth]{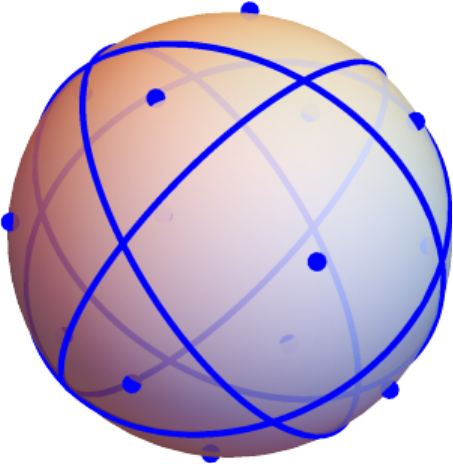}
\label{hhh}\hspace{.7cm}}
\caption{Hybrid $5$-designs.
}\label{fig:summary}
\end{minipage}

\end{figure}
\begin{figure}
\begin{minipage}[t]{\textwidth}
\subfigure[dodecahedron/\newline icosahedron]{
\includegraphics[width=.22\textwidth]{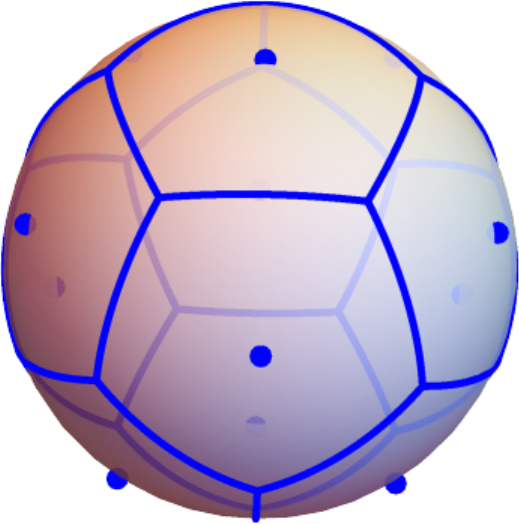}
\label{g}}
\subfigure[icosahedron/ \newline dodecahedron]{
\includegraphics[width=.22\textwidth]{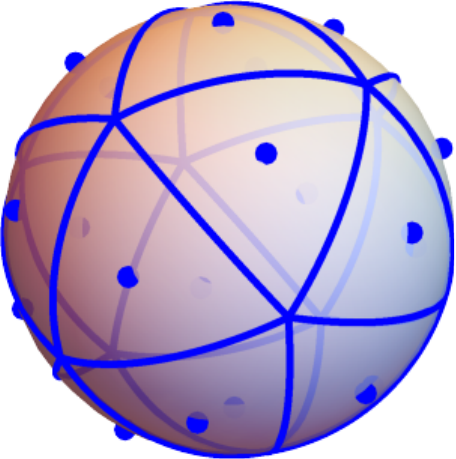}
\label{h}}
\subfigure[rhombic triacontahedron/icosidodecahedron]{
\includegraphics[width=.22\textwidth]{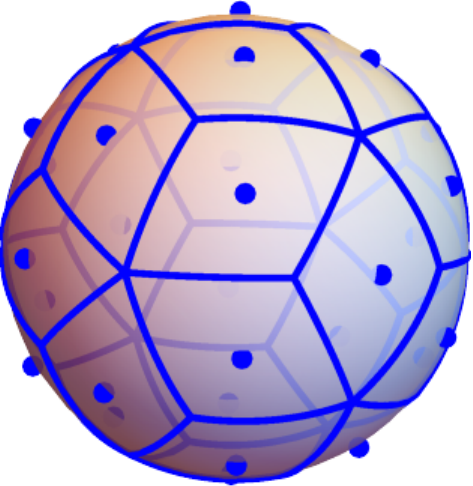}
\label{gg}\hspace{.3cm}}%
\subfigure[icosidodecahedron/ \newline  rhombic triacontahedron]{
\includegraphics[width=.22\textwidth]{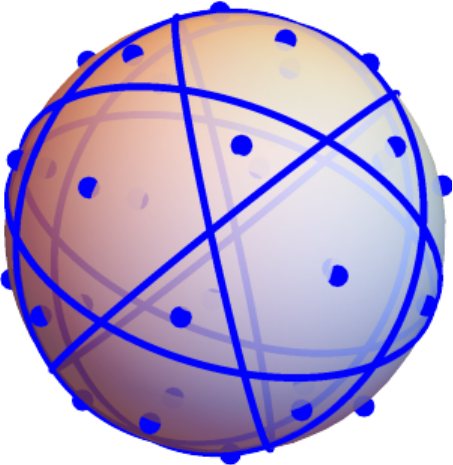}
\label{hh}\hspace{1cm}}
\caption{Hybrid $9$-designs.
}\label{fig:summary2}
\end{minipage}
\end{figure}

\item[(C)] As our strongest design, we obtain a hybrid $19$-design on $\S^3$ from group orbits 
that stands out from most of its competitors. To the best of our knowledge, the only competing construction is presented in \cite{Goethals:81b}, which provides spherical \( 19 \)-design points. 
These two constructions produce the strongest spherical designs by analytic expressions, to our knowledge. In contrast, other spherical designs of similar or greater strength reported in the literature are typically obtained through numerical procedures, either on \(\mathbb{S}^2\) or \(\mathbb{S}^3\), see \cite{Graf:2011lp,HardinSloane96,Womersley:2018we} and the associated websites\footnote{\url{https://www-user.tu-chemnitz.de/~potts/workgroup/graef/quadrature/index.php}\\ \url{http://neilsloane.com/sphdesigns/}\\ \url{https://web.maths.unsw.edu.au/~rsw/Sphere/EffSphDes/}}. These designs lack explicit analytic expressions and have not been rigorously proven to satisfy the conditions of a spherical \( 19 \)-design. The existence results about asymptotically optimal spherical designs in \cite{Bondarenko:2011kx, Bondarenko:2015eu} are nonconstructive and hence do not provide any explicit designs.

\end{itemize}

\subsubsection*{Outline:} 
In Section \ref{sec:2} we recall the concept of $t$-design curves, geodesic $t$-design cycles, and investigate the integration properties of group orbits of spherical arcs. 
New $t$-design curves are derived from convex polytopes in Section \ref{sec:chains polytopes}. Sections \ref{sec:mixed} and \ref{sec:6} introduce the concept of hybrid $t$-designs and present a construction principle from pairs of dual polytopes. We construct a hybrid $19$-design in Section \ref{sec:Sd}. In Section \ref{sec:no transitivity}, we remove the transitivity assumptions on the polytopes. 

The appendix contains in Section \ref{sec:covering proof} the proof of a lemma, included separately to maintain the flow of the main text. Section \ref{sec:table} contains a table listing relevant properties of some edge-transitive polytopes.

\section{Curves, designs, and group orbits}\label{sec:2}
We begin by introducing the concept of spherical $t$-design curves, as originally proposed in \cite{EG:2023}. Next, we discuss integration along orbits of spherical arcs, which forms the foundation of our further investigations.
\subsection{Spherical design curves}
Here, a curve is defined as a continuous, piecewise smooth mapping \(\gamma : [0,L] \to \mathbb{S}^d\) that is closed in the sense that $\gamma(0)=\gamma(L)$. Given a continuous function $p:\S^d\rightarrow \R$, the path integral along $\gamma$ is 
\begin{equation*}
\int_\gamma p 
=\int_0^{L} p (\gamma(\alpha)) \|\dot{\gamma}(\alpha)\| \, \mathrm{d}\alpha\,,
\end{equation*}
and its arc-length is $\ell(\gamma) = \int_\gamma 1$. For $t\in\N$, recall that \( \Poltd \) denotes the space of all polynomials of total degree at most \( t \) in \( d+1 \) variables, restricted to the sphere. We follow \cite{EG:2023} and call $\gamma$ a spherical $t$-design curve if 
\begin{equation*}
\frac{1}{\ell(\gamma)} \int_\gamma p  = \int_{\S^d} p  \, , \quad \text{for all } p  \in \Poltd\,.
\end{equation*}
Integration over the entire sphere \(\mathbb{S}^d\) is always taken with respect to the normalized, orthogonally invariant measure, and 
the parameter $t$ is referred to as the strength of the design (curve). Unlike \cite{EG:2023}, we permit arbitrary self-intersections, so that the curve may traverse sections multiple times.

A piecewise smooth curve can be constructed by connecting finitely many points with geodesic arcs. To explicitly parametrize a geodesic arc, recall that the geodesic distance between points \(x\) and \(y\) on \(\S^d\) is given by  
\[
\dist(x, y) = \arccos \langle x, y \rangle\in[0,\pi]\,.
\]  
For \(y \neq \pm x\), the distance is neither $0$ nor $\pi$, and the geodesic arc is parametrized as \(\gamma_{x, y}: [0, \dist(x,y)] \to \S^d\),  
\[
\gamma_{x, y}(s) = \frac{\sin(\dist(x, y)-s)}{\sin(\dist(x,y))} \, x + \frac{\sin(s)}{\sin(\dist(x,y))}\, y\,.
\]  
Integration along a geodesic arc is direction-independent, so that $\int_{\gamma_{x,y}}p  =\int_{\gamma_{y,x}}p $. 

By connecting points \(x_0, \ldots, x_n \in \S^d\) with such geodesic arcs, one obtains a geodesic chain, analogous to a polygonal chain in Euclidean space. For closed chains, called cycles, we set \( x_{n+1} := x_0 \). The integral over a cycle \(\gamma = (\gamma_{x_j, x_{j+1}})_{j=1}^n\) is  
\[
\int_\gamma p  = \sum_{j=1}^n \int_{\gamma_{x_j, x_{j+1}}} \!\!\!p \,,
\]  
and the length of the cycle is the sum of the arc lengths $\ell(\gamma) = \sum_{j=1}^n \dist(x_j, x_{j+1})$. 

Every geodesic cycle \(\gamma\)  induces a curve and is called a spherical \( t \)-design cycle (or $t$-design curve) if  
\[
\frac{1}{\ell(\gamma)} \int_\gamma p  = \int_{\S^d} p  \quad \text{for all } p  \in \Poltd\,.
\]  
In the following example we construct spherical $t$-design cycles derived from a set of great circles.
\begin{example}[Figures \ref{fig:octa standard}, \ref{fig:cubo intro}, \ref{ggg0}: arrangement of great circles in $\mathbb{S}^2$]\label{ex:gc}
Consider a spherical $t$-design on $\S^2$ of $2n$ pairwise antipodal points \cite{Hoggar:1982fk}.  
Each line through antipodal points serves as the normal of a hyperplane, whose intersection with the sphere $\S^2$ is a great circle. Hence, we end up with a set of $n$ circles, which may serve as the trace of a curve.

To define this curve, we construct a graph whose vertices are the intersection points of these great circles, and whose edges are segments of the great circles connecting the vertices. An Euler cycle in this graph is a cycle that visits every edge exactly once. By Euler's Theorem on undirected, connected graphs, such a cycle exists if and only if each vertex has an even degree \cite{Wilson:1998qa}. In our case, each vertex is attached to an even number of edges. Thus, the construction ensures the existence of an Euler cycle. This Euler cycle is a geodesic cycle on $\mathbb{S}^2$, and the results in \cite{EG:2023} imply that it constitutes a $t$-design curve. 

For instance, the $6$ vertices of the octahedron and the $8$ vertices of the cube each form a $3$-design with pairwise antipodal points. The induced great circles lead to the trace of the edges of the spherical octahedron and the cuboctahedron as shown in Figures \ref{fig:octa standard} and \ref{fig:cubo intro}, respectively. The $12$ vertices of the icosahedron are antipodal and form a $5$-design. The associated $6$ circles yield the trace of the spherical icosidodecahedron in Figure \ref{ggg0}.
\end{example}

More explicit constructions of geodesic \( t \)-design cycles are provided in \cite{EGK:24}. The present work focuses on constructions using group orbits.  

\subsection{Orbits of finite subgroups of the orthogonal group}\label{sec:orbits}
A finite subgroup $G$ of the orthogonal group $\mathcal{O}(d+1)$ is called $t$-homogeneous in \cite{Bannai84,Harpe:2004} if every orbit \( X =Gx_0= \{g x_0 : g \in G\} \), for  $x_0 \in \mathbb{S}^d$, forms a spherical \( t \)-design. 
We will make use of the finite Coxeter groups $A_{d+1}$, $B_{d+1}$, $D_{d+1}$, and the exceptional groups $H_3$, $F_4$, $H_4$, $E_6$, $E_7$, and $E_8$ \cite{Kane}. These are precisely the finite reflection groups in $\mathcal{O}(d+1)$, for $d\geq 2$, and the respective value of $t$ is provided in \cite{Harpe:2004}, see also Table \ref{tab:list0} and Table \ref{tab:list} in the appendix. A more detailed analysis of $H_4$ and $E_8$ in regard to spherical designs is contained in \cite{Nozaki}. 

To expand exact integration from point orbits to orbits of geodesic arcs, we consider the space of $G$-invariant polynomials of degree at most $t$ given by 
\[
\PoltdG := \{p  \in \Poltd : p \circ g = p  \text{ for all } g \in G\}\,.
\]  
It is already observed in \cite{Goethals:81b,Harpe:2004} that the group $G$ is $t$-homogeneous if and only if $\PoltdG$ consists solely of constant functions. In other words, $\PoltdG$ is one-dimensional, commonly denoted by \( \PoltdG = \mathbb{R} \). 

 This formulation enables us to derive exact integration for orbits of geodesic arcs.

\begin{lemma}\label{prop:orbit and arcs}
If $G\subseteq\mathcal{O}(d+1)$ is $t$-homogeneous and $\gamma_0\subseteq\sd$ is a geodesic arc, then the orbit $(g \gamma_0)_{g\in G}$ induced by $\gamma_0$ satisfies the exact integration condition
\begin{equation*}
\frac{1}{|G|}\sum_{g\in G} \frac{1}{\ell(\gamma_0)} \int_{g\gamma_0} p  = \int_{\sd} p ,\,\quad\text{for all $p \in\Poltd$.}
\end{equation*}
\end{lemma}
If $\gamma_0$ is a constant curve, so that $\gamma_0(s)=x_0\in\sd$, for $s\in[0,1]$, then the integration reduces to a point evaluation $\frac{1}{\ell(\gamma_0)} \int_{g\gamma_0} p =p (gx_0)$. Thus, each orbit $Gx_0$ is a $t$-design, as mentioned above and used in \cite{McLaren63}. 
\begin{proof}
For $p \in\Poltd$, we denote its average over $G$ by the Reynolds operator 
\begin{equation}\label{eq:average}
p ^G:=\frac{1}{|G|}\sum_{g\in G} p \circ g\,.
\end{equation}
This average leads to the compact expression
\begin{align*}
\frac{1}{|G|}\sum_{g\in G} \frac{1}{\ell(\gamma_0)} \int_{g\gamma_0} p  & = \frac{1}{|G|}\sum_{g\in G} \frac{1}{\ell(\gamma_0)} \int_{\gamma_0} p \circ g
 = \frac{1}{\ell(\gamma_0)} \int_{\gamma_0} p ^G\,.
\end{align*}
Since $p^G$ is $G$-invariant and $G$ is $t$-homogeneous, $p ^G$ is constant and we obtain
\begin{equation*}
\frac{1}{\ell(\gamma_0)} \int_{\gamma_0} p ^G=\int_{\sd}p ^G\,. 
\end{equation*}
The orthogonal invariance of the surface measure on the sphere $\sd$ leads to $\int_{\sd}p \circ g=\int_{\sd}p $, for all $g\in G$, and hence $\int_{\sd}p ^G=\int_{\sd}p $, which concludes the proof. 
\end{proof}

Starting with an arbitrary spherical arc \(\gamma_0\), the orbit of geodesic arcs \(\Gamma = \{g\gamma_0 : g \in G\}\) is generally not a connected set. As a result, it does not form a geodesic cycle. To overcome this issue, we will soon introduce additional structure to the orbit \(\Gamma\). This structure is provided by convex polytopes, which are discussed next.

\section{Spherical $t$-design cycles from convex polytopes}\label{sec:chains polytopes}
A convex polytope \(\mathcal{P}\) in \(\mathbb{R}^d\) is defined as the convex hull of a finite set of points in \(\mathbb{R}^d\). We assume \(\mathcal{P}\) is centered at the origin and scaled such that it lies within the unit ball. By projecting outward, each edge of the polytope defines a geodesic arc on the unit sphere. On the two-sphere $\S^2$, for instance, those arcs yield a spherical tiling as studied in \cite{Gruenbaum2} in more generality.

To construct a geodesic cycle on the sphere, we use that the vertices and edges of the polytope define a graph. An Euler cycle traverses each edge of the graph exactly once, and according to Euler's theorem, a connected graph has an Euler cycle if every vertex is incident with an even number of edges, cf.~Example \ref{ex:gc}. By outward projection, such an Euler cycle translates directly into a geodesic cycle on the sphere.

If not all vertices are incident to an even number of edges, we double all edges of the graph. This modified graph ensures that all vertices now meet the conditions for admitting an Euler cycle.

\subsection{Euler cycles of convex polytopes induce spherical $t$-design cycles}
The symmetries of a convex polytope \(\mathcal{P}\) in $\R^d$ are the elements \(g \in \mathcal{O}(d)\) that leave \(\mathcal{P}\) invariant, i.e., \(g\mathcal{P} = \mathcal{P}\). These elements constitute the (full) symmetry group \(G\) of \(\mathcal{P}\), which also acts on the geometric features of the polytope. 

One such feature is a face and defined as the intersection of the polytope with a hyperplane, so that the polytope lies on one side of the hyperplane. Faces can range from vertices ($0$-dimensional faces) and edges ($1$-dimensional faces) all the way up to facets ($d-1$-dimensional faces) and the polytope itself ($d$-dimensional face). 

A (proper) flag is a tuple \(f = (f_0, f_1, \ldots, f_{d-1})\) of nested faces of the polytope, so that each \(f_k\) is \(k\)-dimensional, and \(f_{k-1} \subseteq f_k\) for \(k = 1, \ldots, d-1\). A convex polytope in \(\mathbb{R}^d\) is called regular if its symmetry group \(G\) acts transitively on its flags. That is, for any two flags \(f\) and \(f'\), there exists \(g \in G\) such that \(f' = gf\).  

In particular, regular convex polytopes are \(k\)-face-transitive for \(k = 0, \ldots, d-1\), i.e., the symmetry group acts transitively on the \(k\)-dimensional faces. Consequently, each set of \(k\)-dimensional faces forms a single orbit under the action of \(G\).   

If the symmetry group \(G\) of a convex polytope is \(t\)-homogeneous, vertex-transitivity (transitivity on \(0\)-dimensional faces) ensures that the vertices, when projected onto the sphere, form a \(t\)-design. We now extend this observation from vertices to edges.
\begin{thm}\label{thm:homogeneous}
Suppose that $\mathcal{P}$ is an edge-transitive convex polytope in $\mathbb{R}^{d+1}$. If its symmetry group $G\subseteq \mathcal{O}(d+1)$ is $t$-homogeneous, then every Euler cycle projected onto the sphere $\mathbb{S}^d$ is a spherical $t$-design cycle.
\end{thm}
If the assumptions of Theorem \ref{thm:homogeneous} are satisfied but $\mathcal{P}$ does not admit any Euler cycles, then as mentioned before we may double each edge and every Euler cycle in this new configuration yields a geodesic $t$-design cycle.

\begin{proof}
Let $(\mathrm{e}_i)_{i=0}^m$ denote all edges of $\mathcal{P}$ ordered according to an Euler cycle and let $\gamma=(\gamma_i)_{i=0}^m$ be the induced geodesic cycle on the sphere, so that $\gamma_i$ is the geodesic arc induced by projecting the edge $\mathrm{e}_i$ onto the sphere. 

Since $G$ acts transitively on the edges, the edges form a single orbit $\{\mathrm{e}_0,\ldots,\mathrm{e}_m\}=\{g\mathrm{e}_0:g\in G\}$. For every $g\in G$, the edge $g\mathrm{e}_0$ corresponds to the geodesic arc $g\gamma_0$. Thus, $G$ also acts transitively on the arcs. Since the cardinality of an orbit divides the group order, we deduce
\begin{equation*}
\frac{1}{m+1}\sum_{i=0}^m \frac{1}{\ell(\gamma_i)}\int_{\gamma_i}p   = \frac{1}{|G|} \sum_{g\in G}\frac{1}{\ell(\gamma_0)}\int_{g\gamma_0}p \,.
\end{equation*}
Lemma \ref{prop:orbit and arcs} implies the integration condition
\begin{equation*}
\frac{1}{m+1}\sum_{0=1}^m \frac{1}{\ell(\gamma_i)}\int_{\gamma_i}p  = \int_{\S^d} p \,, \quad\text{for all $p \in\Poltd$,}
\end{equation*}
so that the Euler cycle projected onto the sphere is a spherical $t$-design cycle. 
\end{proof}

The class of edge-transitive polytopes has not yet been fully characterized \cite{GruenbaumBook,Martini94,WinterPhD}, but several edge-transitive convex polytopes are known and subclasses have been studied \cite{Winter00}.

\subsection{Edge-transitive convex polytopes in $\R^3$}
All edge-transitive convex polytopes in $\R^3$ are known \cite{Winter23}, and there are $9$ of them, see Table \ref{tab:list0}. Five are the regular convex polytopes, namely the Platonic solids. These are the tetrahedron, the octahedron, the cube, the icosahedron, and the dodecahedron. Their symmetries are the tetrahedral group  $A_3$, the octahedral group $B_3$, and the icosahedral group $H_3$, that are $t$-homogeneous for $t=2,3,5$, respectively.  
The other four edge-transitive convex polytopes are the cuboctahedron, the icosidodecahedron, the rhombic dodecahedron, and the rhombic triacontahedron. The latter two are the only ones that are not vertex-transitive. 

Table \ref{tab:list0} provides a summary of the relevant properties of those nine polytopes, see also Table \ref{tab:list} in the appendix. The induced geodesic $t$-design cycles as claimed in Theorem \ref{thm:homogeneous} are shown in Figures \ref{fig:summary00000}, \ref{fig:summary0000}, and  \ref{fig:summary000} in the introduction.

\begin{table}
\begin{tabular}{|c|c|c|c|c|}
\hline
polytope & symmetry & $t$ &  $\#$ vertices &  $\#$ edges   \\
\hline
\hline
tetrahedron & $A_3$ & $2$ & $4$ &   $6$\\
\hline
octahedron& $B_3$ & $3$ & $6$&  $12$\\
cube & $B_3$ & $3$ & $8$ &  $12$\\
cuboctahedron & $B_3$ & $3$ & $12$ & $24$\\
rhombic dodecahedron & $B_3$ & $3$ & $14$  &  $24$ \\
\hline
icosahedron & $H_3$ & $5$ & $12$ &  $30$ \\
dodecahedron & $H_3$ & $5$  & $20$ & $30$  \\
icosidodecahedron & $H_3$ & $5$  & $30$ & $60$ \\
rhombic triacontahedron& $H_3$ & $5$ & $32$ & $60$ \\
\hline
\end{tabular}
\vspace{1ex}
\caption{List of all edge-transitive convex polytopes in $\R^3$ that induce spherical $t$-design curves. 
}\label{tab:list0}
\end{table}

\subsection{Further edge-transitive polytopes}
In dimension $d\geq 4$, each edge-transitive convex polytope is also vertex-transitive \cite{Winter23}. 
We provide a list of some edge-transitive convex polytopes in $\mathbb{R}^d$ in Table \ref{tab:list} in the appendix. All regular convex polytopes in $\mathbb{R}^d$ are edge-transitive. Among these, three types exist in every dimension: the $d$-dimensional regular simplex (often referred to as the $d$-dimensional tetrahedron), the $d$-dimensional cube, and the $d$-dimensional octahedron.

The dimension $d=4$ is unique, as $\mathbb{R}^4$ contains six convex regular polytopes. In addition to the $4$-tetrahedron ($5$-cell), the $4$-octahedron ($16$-cell), and the $4$-cube ($8$-cell), there are the hypericosahedron ($600$-cell), the hyperdodecahedron ($120$-cell), and the $24$-cell. 

For a convex regular polytope $\mathcal{P}$, its rectified polytope $\mathcal{P}_{\mathrm{rect}}$ is defined as the convex hull of the midpoints of its edges, and they form another subclass of edge-transitive polytopes \cite{WinterPhD}.  

A notable further subclass of edge-transitive polytopes are the distance-transitive convex polytopes, as introduced in \cite{Winter00}. Examples include the $2_{21}$-polytope in $\mathbb{R}^6$ and the $3_{21}$-polytope in $\mathbb{R}^7$ \cite[Theorem 5.10]{Winter00}. The $2_{21}$-polytope connects each vertex by $16$ edges, allowing Euler cycles, while the $3_{21}$-polytope connects vertices with $27$ edges, requiring edge doubling. Their symmetry groups, $E_6$ and $E_7$, are $4$- and $5$-homogeneous \cite{Harpe:2004} and enable the construction of spherical $4$- and $5$-design cycles, respectively. The $d$-demi-cube in $\R^d$ is also distance-transitive and its symmetry group $D_d$ is $3$-homogeneous. 

The \( 4_{21} \)-polytope in \( \mathbb{R}^8 \), also known as the \( E_8 \) root polytope, is edge-transitive \cite[Proposition 4]{WL21} but not distance-transitive \cite[Theorem 5.10]{Winter00}. Its symmetry group, \( E_8 \), is \( 7 \)-homogeneous \cite{Harpe:2004}. Consequently, the family of \( k_{21} \)-polytopes in \( \mathbb{R}^{4+k} \) for \( k = 0, \dots, 4 \) is edge-transitive. Here, \( 0_{21} \) is the rectified 4-simplex, and \( 1_{21} \) is the 5-demi-cube, both of which have been mentioned earlier in this section.

For a brief summary and few more details, we refer to Table \ref{tab:list} in the appendix.

\section{Hybrid spherical designs}\label{sec:mixed}
The $t$-design curves discussed in the previous sections are based on edge-transitive polytopes. This approach utilizes their $t$-homogeneous symmetry groups and produces several new $t$-design curves. However, the achievable range of the degree $t$ is limited. The highest value of $t$ in Table \ref{tab:list} is $11$ and corresponds to the symmetry group $H_4$ of the $120$-cell and $600$-cell. The examples in $\R^3$ are restricted to the degrees $t = 2, 3, 5$.  

To achieve higher degrees, we introduce the concept of hybrid designs that combine a spherical (design) curve with additional (design) points.

\subsection{Introduction to hybrid designs}
In sampling theory, the path of a curve can model a mobile sensor, while individual points correspond to static sensors. To improve coverage and integration capabilities, we now combine curves and points. In particular, we may balance the contributions of static and mobile sensors by an additional factor $\beta$. 
\begin{definition}[hybrid $t$-designs]\label{def:mixed}
Given a curve $\gamma$ on $\sd$ and a finite set $X\subseteq\sd$, we call the pair $\left(\gamma,X\right)$ a hybrid $t$-design if there is a constant $0\leq \beta\leq 1$ such that 
\begin{equation}\label{eq:def mixer}
\frac{\beta}{|X|} \sum_{x\in X} p (x) + \frac{1-\beta}{\ell(\gamma)}  \int_{\gamma} p = \int_{\sd} p \,,\qquad \text{for all } p \in\Poltd\,.
\end{equation}
\end{definition}
Evaluation for $p \equiv 1$ implies that the contributions of the points and the curve need to be balanced by an affine combination, which leads to the factors $\beta$ and $1-\beta$ in \eqref{eq:def mixer}. In the context of numerical quadrature, it is reasonable to require that both, the points and the curve, contribute positively. This leads to the condition \(0 < \beta < 1\). If $\gamma$ is a $t$-design curve and $X\subseteq \sd$ are $t$-design points, then $\left(\gamma,X\right)$ is a hybrid $t$-design for every $\beta$. The artificial choice $\beta=0$ or $\beta=1$ is associated to pure designs that consist of either a curve or points exclusively. The balancing factor offers a degree of freedom not available in pure designs. 

The integration formula can be interpreted as being induced by a probability measure on \(\mathbb{S}^d\) whose support consists of points and the trajectory of a curve. For general probability measures, we now derive a necessary covering property that provides insight into how the points should be positioned relative to the curve. The covering radius of a closed subset \( M \subseteq \mathbb{S}^d \) is defined as  
\[
\delta(M) := \sup_{x \in \mathbb{S}^d} \min_{y \in M} \operatorname{dist}(x, y),
\]  
representing the radius of the largest open ball in \( \mathbb{S}^d \) that does not intersect \( M \). 

While explicit upper bounds on the covering radius of spherical design points, particularly for small \(t\), are derived in \cite{Sole}, the following result illustrates that the curve and the points of a hybrid design must also cover the sphere sufficiently well.
\begin{lemma}\label{prop:cover}
There is a constant $C_d>0$, depending on the dimension $d$, such that the following holds: If $\mu$ is a probability measure on $\sd$ such that 
\begin{equation}\label{eq:exact integration 0}
\int_{\sd} p  (x) \mathrm{d}\mu(x) = \int_{\sd} p \,,\quad\text{for all $p \in\Poltd$,}
\end{equation}
then the covering radius of the support of $\mu$ satisfies
\begin{equation*}
\delta(\supp(\mu)) \leq C_d t^{-1}\,.
\end{equation*}
\end{lemma}
For probability measures with finite support, this result is contained in \cite{Reimer} and shows that the upper bounds in \cite{Sole} do not have the optimal asymptotics for large $t$. The more general claim in Lemma \ref{prop:cover} might also be well-known to the experts (see \cite{marzo07,marzo08} for related results). For instance, it can derived from the lines of the first part of the proof in \cite[Theorem 2.2]{EG:2023}, and we provide the details in the Appendix \ref{sec:covering proof} for the sake of completeness.

Since the constant \( C_d \) is not specified here, Lemma \ref{prop:cover} does not offer a direct bound on the covering radius for specific hybrid \( t \)-designs. It rather suggests that a smaller covering radius is preferable, meaning points should be placed to fill gaps where the curve has large voids. We continue with a few elementary hybrid designs in Figures \ref{a}, \ref{mixed:b}, and \ref{mixed:c} that illustrate this idea.
\begin{proposition}\label{prop:1-2-3}
The following configurations form hybrid designs:
\begin{itemize}
\item[\textnormal{(i)}] \textnormal{Figure \ref{a}:} a circle parallel to the equator at height 
$\frac{1}{3}$ paired with the south pole $u_{S}=(0,0,-1)$ forms a hybrid $2$-design with $\beta=\frac{1}{4}$. 
\item[\textnormal{(ii)}] \textnormal{Figure \ref{mixed:b}:} the south pole $u_{S}$ and a regular spherical triangle centered around the north pole $u_N=(0,0,1)$ at the correct height form a hybrid $2$-design. 

\item[\textnormal{(iii)}] \textnormal{Figure \ref{mixed:c}:} the equator paired with the north pole $x_N$ and the south pole $u_S$ yield a hybrid $3$-design with $\beta=\frac{1}{3}$. 

\end{itemize}
\end{proposition}

\begin{proof}
It is sufficient to check the integration conditions for the monomials. We used the computer algebra system Mathematica \cite{mathematica} to streamline the proof, but all symbolic calculations remain verifiable by hand. 

The orthogonal invariance of the surface measure of $\S^2$ yields that the integral $\int_{\S^2} x^k y^l z^m $ vanishes if at least one of the exponents $k,l,m\in\N$ is odd. Moreover, the identity $x^2+y^2+z^2=1$ leads to $\int_{\S^2}x^2=\int_{\S^2}y^2=\int_{\S^2}z^2=\frac{1}{3}$. 

(i) A circle at height $0\leq r\leq 1$ is parametrized by $\gamma_r:[0,1]\rightarrow \S^2$,
\begin{equation}\label{eq:curve gammar}
\gamma_r(\alpha)=
\begin{pmatrix}
\cos(\alpha)\sqrt{1-r^2}\\
 \sin(\alpha)\sqrt{1 - r^2}\\
 r
\end{pmatrix}\,.
\end{equation}
We observe $\int_{\gamma_r}x=\int_{\gamma_r}y=0$ and the respective point evaluations at the south pole also vanish. We now let the balancing factor $\beta(r)$ depend on $r$. Since $\frac{1}{\ell(\gamma_r)}\int_{\gamma_r}z = r$, to match $\int_{\S^2}z=0$, we require $-\beta(r) + (1-\beta(r))r=0$. Thus, $\beta(r) = \frac{r}{1+r}$ yields that $(\gamma_r,\{u_S\})$ is a hybrid $1$-design. 

For $r=\frac{1}{3}$, we obtain $\beta(r)=\frac{1}{4}$ and $\frac{1}{\ell(\gamma_r)}\int_{\gamma_r}z^2 = \frac{1}{9}$. The sum of the point evaluation and the path integral of the monomial $z^2$ is $\frac{1}{4}+(1-\frac{1}{4})\frac{1}{9} = \frac{1}{3}$, which matches $\int_{\S^2}z^2=\frac{1}{3}$. The symmetry of the configuration and $x^2+y^2+z^2=1$ eventually imply that the analogous evaluation for the monomials $x^2$ and $y^2$ also yields $\frac{1}{3}$, so that we obtain a hybrid $2$-design.

(iii) If $r=0$ in \eqref{eq:curve gammar}, then the pair $(\gamma_0,\{u_N,u_S\})$ is antipodal, and hence the exact integration identities are satisfied when the sum $k+l+m$ is odd. It remains to check the monomials $xy,xz,yz$, and $x^2,y^2,z^2$. Direct calculations based on symmetries reveal that $(\gamma_0,\{u_N,u_S\})$ `integrates' $xy,xz,yz$ to zero. To match $\int_{\S^2}z^2=\frac{1}{3}$, we may simply choose $\beta=\frac{1}{3}$ since the equator does not contribute in that case. For the remaining cases $x^2$ and $y^2$, the poles do not contribute and we obtain 
\begin{equation*}
\frac{2}{3}\frac{1}{\ell(\gamma_0)} \int_{\gamma_0}x^2=\frac{2}{3}\frac{1}{\ell(\gamma_0)} \int_{\gamma_0}y^2 = \frac{1}{3}\,.
\end{equation*}
Thus, $(\gamma_0,\{u_N,u_S\})$ is a hybrid $3$-design.

(ii) This configuration is more elaborate. For $a\in(0,\frac{\pi}{2}]$, we parametrize the triangle around the north pole by the three control points
\begin{equation}\label{eq:triangle points}
u_1=\left(\begin{smallmatrix}
\sin(a)\\0\\\cos(a)
\end{smallmatrix}\right),
\qquad
u_2=\left(\begin{smallmatrix}
-\frac{1}{2}\sin(a)\\
\frac{\sqrt{3}}{2}\sin(a)\\
\cos(a)
\end{smallmatrix}\right),
\qquad u_3 = \left(\begin{smallmatrix}
-\frac{1}{2}\sin(a)\\
-\frac{\sqrt{3}}{2}\sin(a)\\
\cos(a)
\end{smallmatrix}\right)\,.
\end{equation}
We now solve the resulting equations for the monomials in the parameter $a$. To match $\int_{\mathbb{S}^2}z=0$, symbolic calculations yield 
 \begin{equation}\label{eq:rho}
\beta(a)=\tfrac{1}{1+\frac{\sqrt{5+3 \cos (2 a)}}{\sqrt{6} \sin(2a)}\arccos\left(\frac{1}{4}+\frac{3}{4} \cos(2a)\right) }\in[0,\tfrac{1}{2}]\,,
\end{equation}
see Figure \ref{fig:rho in a}. 
This also matches $\int_{\mathbb{S}^2}x=\int_{\mathbb{S}^2}y=0$, so that the south pole paired with the triangle forms a hybrid $1$-design for the balancing factor $\beta(a)$. 

\begin{figure}
\includegraphics[width=.4\textwidth]{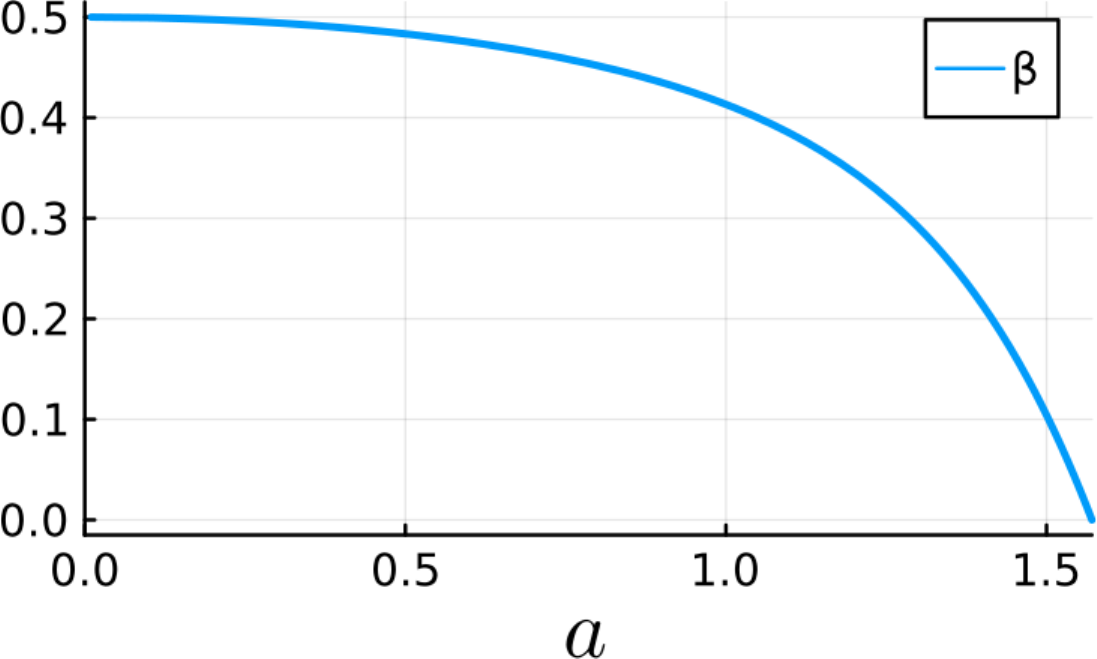}
\caption{The factor $\beta(a)$ in \eqref{eq:rho} balances the contribution of the south pole with the triangle at height $\cos(a)$ in the proof of Part (ii) of Proposition \ref{prop:1-2-3}. We obtain a hybrid $2$-design at $\hat{a}\approx 1.359$ with $\beta(\hat{a})\approx 0.249$.}\label{fig:rho in a}
\end{figure}

To derive a hybrid $2$-design, we must now select $a$. We denote the triangle by $\gamma_a$. For $i=1,2,3$, define the continuous functions
\begin{equation*}
E_i(a):= \beta(a) p _i(0,0,-1) + \frac{1-\beta(a)}{\ell(\gamma_a)} \int_{\gamma_a} p _i\,,
\end{equation*}
where $p _1$, $p _2$, and $p _3$ are the monomials $x^2$, $y^2$, and $z^2$, respectively. 
Symmetry leads to the identity $E_1(a)=E_2(a)$, for all $a\in(0,\frac{\pi}{2})$. To identify $\hat{a}\in(0,\frac{\pi}{2})$ such that $E_1(\hat{a})=E_3(\hat{a})$, we prepare to apply the mean value theorem. 

For $a=\frac{\pi}{2}$, the spherical triangle $\gamma_a$ coincides with the equator, and $a \rightarrow 0$ deforms the triangle $\gamma_a$ into the north pole. We observe $\beta(0)=\frac{1}{2}$ and $\beta(\frac{\pi}{2})=0$, so that 
\begin{align*}
\lim_{a\searrow 0}E_1(a) &= 0, \qquad\quad E_1(\tfrac{\pi}{2}) >0\,,\\
\lim_{a\searrow 0}E_3(a)  &= 1, \qquad \quad E_3(\tfrac{\pi}{2}) = 0\,, 
\end{align*}
and $E_1-E_3$ changes sign on the interval $(0,\frac{\pi}{2})$. The mean value theorem implies that there is $\hat{a}\in(0,\frac{\pi}{2})$ such that $E_1(\hat{a})=E_3(\hat{a})$. 

Thus, we know that $E_1(\hat{a})=E_2(\hat{a})=E_3(\hat{a})$ and the identity $x^2+y^2+z^2=1$ leads to $E_i(\hat{a})=\frac{1}{3}$ for $i=1,2,3$, which matches 
\begin{equation*}
\int_{\mathbb{S}^2}x^2=\int_{\mathbb{S}^2}y^2=\int_{\mathbb{S}^2}z^2=\frac{1}{3}\,.
\end{equation*}

Thus, the pair $(\gamma_{\hat{a}},\{u_S\},)$ integrates the monomials $x^2,y^2,z^2$ exactly. We also observe symbolically (can also be argued by symmetry) that the monomials $xy, xz,yz$ integrate to zero, so that we obtain a hybrid $2$-design. Numerically, we observe $\hat{a}\approx 1.359$. 
\end{proof}

\subsection{Hybrid designs from pairs of dual polytopes}
To improve the covering radius of a geodesic \( t \)-design cycle derived from the edges of a convex polytope in \( \mathbb{R}^3 \), one can include the midpoints of the polytope's 2-dimensional faces, projected onto the sphere. 
More formally, for a regular convex polytope \( \mathcal{P} \) in \( \mathbb{R}^3 \), its dual polytope \( \mathcal{P}_{\mathrm{dual}} \) is formed by taking the convex hull of the centers of \( \mathcal{P} \)'s facets (its 2-dimensional faces). This process results in another regular convex polytope, and \( (\mathcal{P}, \mathcal{P}_{\mathrm{dual}}) \) is referred to as a dual pair. The tetrahedron is special, as it is self-dual, meaning its dual realization is a rotated version of itself.

For such a dual pair, the geodesic cycle is induced by the edges of the first (primal) polytope, while vertices originate from the second (dual) polytope. 
If the associated pure designs have strength $t$, then we expect that their hybrid use leads to a stronger hybrid $s$-design with $s>t$. The following theorem confirms this idea for the 
hybrid designs that are illustrated in Figures \ref{fig:summary0}, \ref{fig:summary}, and \ref{fig:summary2} in the introduction except for Figures \ref{hhh} and \ref{hh}. 

\begin{thm}\label{thm:main}
Table \ref{table:Theorem} provides a list of primal and dual polytopes, whose edges and vertices yield hybrid $s$-designs on $\S^2$ with the specified balancing factor $\beta$.
\end{thm}

\renewcommand{\arraystretch}{1.5}
\begin{table}
\begin{tabular}{|c|c|r|}
\hline
$\mathcal{P}/\mathcal{P}_{\mathrm{dual}}$& $s$ & $\beta$\\
\hline\hline
tetrahedron& $3$ & $ \tfrac{4}{4 + 3\sqrt{2} \arccos(-\frac{1}{3})} \approx 0.33$\\
\hline
octahedron/cube & $5$ & $\tfrac{9}{25}=0.36$\\
cube/octahedron & $5$ & $\tfrac{10 \sqrt{2} + 3 \arccos(\frac{1}{3})}{10 \sqrt{2} + 35 \arccos(\frac{1}{3})} \approx 0.31$\\
rhombic dodecahedron/cuboctahedron & $5$ & $\tfrac{10 - 3\sqrt{2} \arccos(\frac{1}{\sqrt{3}})}{10 + 5 \sqrt{2} \arccos(\frac{1}{\sqrt{3}})} \approx 0.35$\\
\hline
dodecahedron/icosahedron & $9$ & $\tfrac{1190 \sqrt{5} - 675 \arccos\left(\frac{\sqrt{5}}{3}\right)}{1190 \sqrt{5} + 6237 \arccos\left(\frac{\sqrt{5}}{3}\right)} \approx 0.30$\\
icosahedron/dodecahedron & $9$ & $\tfrac{126 + 45 \arccos(\frac{1}{\sqrt{5}})}{126 + 301 \arccos(\frac{1}{\sqrt{5}})} \approx 0.38$\\
{\small rhombic triacontahedron/icosidodecahedron} & $9$ & $\tfrac{7(17\sqrt{5}-27)+135 \arccos\left(\sqrt{\frac{1}{3}+\frac{2}{3\sqrt{5}}}\right)}{7(17\sqrt{5}-27)+567\arccos\left(\sqrt{\frac{1}{3}+\frac{2}{3\sqrt{5}}}\right) } \approx 0.37$\\
\hline
\end{tabular}
\caption{Hybrid $s$-designs of Theorem \ref{thm:main}. The primal polytope induces the geodesic cycle through its edges and the dual polytope provides the vertices.}
\label{table:Theorem}
\end{table}
\renewcommand{\arraystretch}{1}

The proof of Theorem \ref{thm:main} relies on the edge-transitivity of the primal polytope and the vertex-transitivity of its dual. The rhombic dodecahedron and rhombic triacontahedron are not vertex-transitive, so they appear in Table \ref{table:Theorem} only as primal polytopes. Their duals, the cuboctahedron and icosidodecahedron, are vertex-transitive.

\section{Proof of Theorem \ref{thm:main}}\label{sec:6}
Our proof of Theorem \ref{thm:main} is guided by the approach in \cite{Goethals:81b}, where suitable group orbits are derived to get strong $t$-design points. The approach consists of two steps and is based on invariant theory. First, we derive an integration result for two group orbits, where one consists of points and the other of spherical arcs. Second, we analyze spherical polynomials that are invariant under the symmetry groups of distinct pairs of dual polytopes. The combination and balancing of these steps will lead to the proof of Theorem \ref{thm:main}.

\subsection{Combining orbits of points with orbits of geodesic arcs}
Lemma \ref{prop:orbit and arcs} enables exact integration on $\Poltd$ using orbits of arcs provided that the space of invariant elements $\PoltdG$ is one-dimensional. 
If we identify $s>t$ such that $\mathrm{Pol}_{\leq s}^G$ is two-dimensional, then we can still derive exact integration of $\Pol_{\leq s}$ by adding another orbit of points.

Assume the orthogonal decomposition  
\[
\PolsdG = \R \oplus \R p _{\invariant}
\]  
is valid, where $p _{\invariant}$ is a suitable $G$-invariant polynomial.  If \( x_0 \in \mathbb{S}^d \) is a root of \( p _{\invariant} \), then the orbit \( Gx_0 \) forms an \( s \)-design, see \cite{Sloane:2003zp}. The subset of the sphere where \( p _{\invariant} \) vanishes has measure zero, so that the random choice of $x_0$ with respect to the surface measure of the sphere fails with probability one. 

We overcome this limitation by introducing an additional orbit of arcs that effectively reverses the situation, making the construction work for almost all points, except for a set of measure zero. Our method parallels McLaren's strategy in \cite{McLaren63}, see also \cite{Goethals:81b}, which employs multiple point orbits for weighted quadrature formulas.

\begin{lemma}\label{thm:basic mixed orbit constr}
Suppose that $G\subseteq \mathcal{O}(d+1)$ is a finite group such that $\PolsdG$ is two-dimensional, so that there is $p _{\invariant}\in\PolsdG$ satisfying the orthogonal decomposition 
$
\PolsdG=\R\oplus \R p _{\invariant}\,.
$ 
\begin{itemize}
\item[\textnormal{(a)}] Then for all arcs $\gamma_0\subseteq \S^d$ and all points $x_0\in \sd $ such that $p_{\invariant} (x_0)\neq \frac{1}{\ell(\gamma_0)}\int_{\gamma_0} p_{\invariant}$, the following integration identity holds,
\begin{equation}\label{eq:int cond orbits}
\frac{\beta}{|G|}\sum_{g\in G}  p(gx) + \frac{1-\beta}{|G|} \sum_{g\in G}\frac{1}{\ell(\gamma_0)}\int_{g\gamma_0} p = \int_{\sd}p\,,\quad\text{for all } p\in\Pol_{\leq s}\,.
\end{equation}
\item[\textnormal{(b)}] Moreover, for all arcs $\gamma_0\subseteq \S^d$, there is $x_0\in\S^d$ such that $ p_{\invariant} (x_0)$ and $\frac{1}{\ell(\gamma_0)}\int_{\gamma_0} p_{\invariant}$ have opposite signs and then \eqref{eq:int cond orbits} holds with a balancing factor that satisfies $0\leq  \beta<1$.
\end{itemize}
\end{lemma}
\begin{proof}
(a) Since $ p_{\invariant} (x_0) - \frac{1}{\ell(\gamma_0)}\int_{\gamma_0}  p_{\invariant}\neq 0$, there is $\beta\in\R$ that satisfies 
\begin{equation}\label{eq:ingredient 2}
\beta\Big( p_{\invariant} (x_0)-\frac{1}{\ell(\gamma_0)}\int_{\gamma_0}  p_{\invariant} \Big) =- \frac{1}{\ell(\gamma_0)}\int_{\gamma_0}  p_{\invariant} \,.
\end{equation}
The polynomial $ p_{\invariant} $ is orthogonal to the constants, so that $ \int_{\sd} p_{\invariant} =0$, and \eqref{eq:ingredient 2} may be stated as
\begin{equation*}
\beta  p_{\invariant} (x_0)+\frac{1-\beta}{\ell(\gamma_0)}\int_{\gamma_0}  p_{\invariant} = \int_{\sd} p_{\invariant} \,.
\end{equation*}

Let $p\in \Polsd$. Since the space of invariant polynomials $\PolsdG=\R\oplus \R p_{\invariant} $ has dimension two, the invariant polynomial $ p^G=\frac{1}{|G|}\sum_{g\in G} p\circ g$ can be written as a linear combination $ p^G=r_0+r_1 p_{\invariant} $ for some $r_0,r_1\in\R$. The summation over the group leads to 
\begin{align*}
\tfrac{\beta}{|G|}\sum_{g\in G}  p(gx_0) + \tfrac{1-\beta}{|G|}\sum_{g\in G}\frac{1}{\ell(\gamma_0)} \int_{g\gamma_0} \!\!p & = 
\beta  p^G(x_0) + \tfrac{1-\beta}{\ell(\gamma_0)} \int_{\gamma_0}  p^G \\
& =  \beta r_0 +\beta r_1   p_{\invariant} (x_0) +(1-\beta) r_0 +
 \tfrac{1-\beta}{\ell(\gamma_0)} r_1\int_{\gamma_0}  p_{\invariant} \\
 & = r_0+ r_1\Big(\beta  p_{\invariant} (x_0) +\tfrac{1-\beta}{\ell(\gamma_0)}\int_{\gamma_0}  p_{\invariant} \Big)\\
 & = r_0+r_1\int_{\sd}  p_{\invariant} \\
 & = \int_{\sd} p^G = \int_{\mathbb{S}^d} p\,,
\end{align*}
where the last equality follows from the orthogonal invariance of the surface measure of the sphere \( \mathbb{S}^d \). This completes the proof of the first part of the theorem.

(b) To verify the second part, we first assume that $\frac{1}{\ell(\gamma_0)}\int_{\gamma_0} p_{\invariant} \neq 0$. Since \( p_{\invariant} \) is orthogonal to the constants, it necessarily assumes both positive and negative values. Thus, there is $x_0 \in \S^d$ such that $ p_{\invariant} (x_0)$ and $\frac{1}{\ell(\gamma_0)}\int_{\gamma_0} p_{\invariant}$ have opposite signs. The choice 
\begin{equation}\label{eq:rho definer}
\beta = \frac{-\frac{1}{\ell(\gamma_0)}\int_{\gamma_0}  p_{\invariant} }{ p_{\invariant} (x_0) - \frac{1}{\ell(\gamma_0)}\int_{\gamma_0}  p_{\invariant} }
\end{equation}
implies $0<\beta<1$ and leads to \eqref{eq:ingredient 2}. Thus, the integration condition \eqref{eq:int cond orbits} holds.

If $\frac{1}{\ell(\gamma_0)}\int_{\gamma_0} p_{\invariant} =0$, then \eqref{eq:int cond orbits} holds for $\beta=0$ and all $x_0\in\S^d$. Alternatively, we may choose $x_0$ as a root of $ p_{\invariant} $ and then \eqref{eq:int cond orbits} holds for all $\beta\in\R$. In any case, this completes the proof of the second part of the theorem.
\end{proof}

To prove Theorem \ref{thm:main} as a consequence of Lemma \ref{thm:basic mixed orbit constr}, we still need some observations from invariant theory, and we will later specify both $x_0$ and $\gamma_0$ such that $ p_{\invariant} (x_0)$ and $\frac{1}{\ell(\gamma_0)}\int_{\gamma_0} p_{\invariant}$ have opposite signs

\subsection{Hybrid designs and invariant polynomials on the sphere}
To apply Lemma \ref{thm:basic mixed orbit constr}, we must check that $\PolsdG=\R\oplus \R p_{\invariant} $. In particular, we seek the maximal $s$ such that $\PolsdG$ is two-dimensional. 

For preparation, we recall the harmonic decomposition of polynomials on the sphere, cf.~\cite{Stein:1971kx}. Let $\mathcal{H}_l$ be the vector space of spherical harmonics of degree $l\in\N$. These are exactly the restrictions to $\S^d\subseteq\R^{d+1}$ of the homogeneous harmonic polynomials of degree $l$ in $d+1$ variables. 
Each space $\mathcal{H}_l$ is invariant under the orthogonal group $\mathcal{O}(d+1)$, and 
the polynomials $\Polsd$ on $\mathbb{S}^d$ are orthogonally decomposed into 
\begin{equation*}
\Polsd = \bigoplus_{l=0}^s \mathcal{H}_l\,.
\end{equation*}
To find the invariant polynomials $\PolsdG=\bigoplus_{l=0}^s \mathcal{H}^G_l$, we can analyze each component $\mathcal{H}_l^G$ individually for $l=0,\ldots,s$.  

The dimensions of the spaces of $G$-invariant polynomials $\mathcal{H}^G_l$ are the coefficients of the (harmonic) Molien series
\begin{equation}\label{eq:Molien coefficients dimensions}
M(\omega):=\sum_{l\in\N} \dim(\mathcal{H}^G_l) \omega^l =\frac{1}{|G|} \sum_{g\in G} \frac{1-\omega^2}{\det(\mathrm{I}_{d+1}-\omega g)}\,,
\end{equation}
cf.~\cite{Goethals:81b,Meyer54}, where the second equality is due to 
Molien's Theorem as derived in \cite{Meyer54} for the spherical harmonics. 
The factor \( 1 - \omega^2 \) appears because we are not dealing with general homogeneous polynomials on $\R^{d+1}$ but spaces of spherical harmonics, cf.~\cite[Section III]{Meyer54}.

For the tetrahedral group $A_3$, the octahedral group $B_3$, and the icosahedral group $H_3$, the Molien series are
\begin{align*}
M_{A_3}(\omega)&=\frac{1}{(1-\omega^{3})(1-\omega^4)} =1+\omega^3+\omega^4+
O(\omega^{6})\,,\\
M_{B_3}(\omega)&=\frac{1}{(1-\omega^{4})(1-\omega^6)} =1+\omega^4+\omega^{6}+O(\omega^{8})\,,\\
M_{H_3}(\omega)&=\frac{1}{(1-\omega^{6})(1-\omega^{10})}=1+\omega^6+\omega^{10}+O(\omega^{12})\,,
\end{align*}
cf.~\cite{Meyer54,Goethals:81b}, see also \cite[Table 1 in Section 3.7, (20) in Section 3.8]{Humphreys}. By extracting the dimensions from the coefficients, we arrive at the following observations.
\begin{lemma}\label{lemma:359}
The spaces \(\mathcal{H}_3^{A_3}\), \(\mathcal{H}_4^{B_3}\), \(\mathcal{H}_6^{H_3}\) are one-dimensional and we have the decompositions  
\begin{equation*}  
\Pol_{\leq 3}^{A_3} = \R \oplus \mathcal{H}_3^{A_3},\qquad   
\Pol_{\leq 5}^{B_3} = \R \oplus \mathcal{H}_4^{B_3}, \qquad
\Pol_{\leq 9}^{H_3} = \R \oplus \mathcal{H}_6^{H_3}\,.
\end{equation*}  
\end{lemma}
The claims of this lemma were already known to Meyer \cite{Meyer54} in 1954, see also Sobolev \cite{Sobolev1962} in 1962 and McLaren \cite{McLaren63} in 1963. They directly imply
\begin{equation*}
\dim(\Pol_{\leq 3}^{A_3})=\dim(\Pol_{\leq 5}^{B_3})=\dim(\Pol_{\leq 9}^{H_3})=2\,.
\end{equation*}
Moreover, these are the highest degrees, $3$, $5$, and $9$, respectively, such that the invariant spaces are two-dimensional. 

According to Lemma \ref{lemma:359}, the assumption of Lemma \ref{thm:basic mixed orbit constr} is satisfied for the respective group $G$ and polynomial degree $s$. There is some $G$-invariant polynomial $ p_{\invariant} \in\PolsdG $ such that $\PolsdG=\R\oplus \R p_{\invariant} $. To compute the balancing factor $\beta$, however, we still need to evaluate and integrate over the basis element $ p_{\invariant} $ for each of the spaces \(\mathcal{H}_3^{A_3}\), \(\mathcal{H}_4^{B_3}\), and \(\mathcal{H}_6^{H_3}\).

\begin{remark}
If the polytope $\mathcal{P}$ is transformed by some orthogonal matrix $O \in \mathcal{O}(3)$, so that the new realization is given by $\tilde{\mathcal{P}} := O\mathcal{P}$, then its symmetry group becomes $\tilde{G} = OGO^*$. The dimensions of the invariant polynomial spaces do not change, but the invariant polynomial $\tilde{p}_{\invariant}$ undergoes the orthogonal transformation $\tilde{p}_{\invariant} =  p_{\invariant} \circ O^*$. Thus, to determine $ p_{\invariant}$, 
we need to specify a particular representation of the symmetry group and implicitly a realization of the associated primal and dual polytopes. 
\end{remark}

For the (full) octahedral group $B_3$, we use the $48$ matrices that permute the three coordinates and an arbitrary number of the coordinates change sign. The vertices of the cube, the octahedron, and the cuboctahedron are orbits of $\frac{1}{\sqrt{3}}(1,1,1)$, $e_1$, and $\frac{1}{\sqrt{2}}(e_1+e_2)$, respectively, where $e_i$ is the $i$-th canonical basis vector taken as a row.

The tetrahedral group $A_3$ consists of $24$ matrices that form a subgroup of $B_3$, where not all but only an even number of coordinates change sign. The vertices of the tetrahedron are the orbit of $\frac{1}{\sqrt{3}}(1,1,1)$. 

The icosahedral group $H_3$ consists of $120$ matrices, generated by the even permutations of the coordinates and an arbitrary number of coordinates change sign, together with the reflection $\mathrm{I}_3-\frac{2}{\|v\|^2} v^\top v$ for $v=(\phi,\phi^{-1},1)$, where $\phi = \tfrac{1+\sqrt{5}}{2}$ is the golden ratio. The vertices of the icosahedron, the dodecahedron, and the icosidodecahedron are orbits of $\frac{1}{\sqrt{1+\phi^2}}(\phi,1,0)$, $\frac{1}{\sqrt{3}}(1,1,1)$, and $e_1$, respectively. 

\begin{proof}[Proof of Theorem \ref{thm:main}]
(i) The polynomial $ p_{\invariant}=p_{3,A_3}$ of degree three invariant under the tetrahedral group $A_3$ is 
\begin{equation*}
p_{3,A_3}(x,y,z)=xyz\,,
\end{equation*}
cf.~\cite{Meyer54,Goethals:81b} and generates $\mathcal{H}_3^{A_3}=\R p_{3,A_3}$. 

A spherical arc $\gamma_0$ is induced by an edge of the primal realization of the tetrahedron that is given by the vertices $\frac{1}{\sqrt{3}}(1,1,1)$, $\frac{1}{\sqrt{3}}(1,-1,-1)$, $\frac{1}{\sqrt{3}}(-1,1,-1)$, and $\frac{1}{\sqrt{3}}(-1,-1,1)$. We also pick a vertex $x_0=\frac{1}{\sqrt{3}}(1,1,-1)$ of the tetrahedron's dual realization. Direct computations lead to 
\begin{equation*}
 p_{3,A_3}(x_0)=-\tfrac{1}{3\sqrt{3}}\,,\qquad\qquad  \frac{1}{\ell(\gamma_0)}\int_{\gamma_0} p_{3,A_3} = \tfrac{2\sqrt{2}}{9\sqrt{3}\arccos(-\frac{1}{3})}\,,
\end{equation*}
and we recognize that the two values have opposite signs. Hence, the assumptions of Lemma \ref{thm:basic mixed orbit constr} are satisfies, and we determine the balancing factor $\beta$ according to \eqref{eq:rho definer} by
\begin{equation*}
\beta = \frac{-\frac{2\sqrt{2}}{9\sqrt{3}\arccos(-\frac{1}{3})}}{-\frac{1}{3\sqrt{3}} -\frac{2\sqrt{2}}{9\sqrt{3}\arccos(-\frac{1}{3})}}
 = \frac{4}{4 + 3\sqrt{2} \arccos(-\frac{1}{3})}\approx 0.33\,.
\end{equation*}

\smallskip
(ii - iv) The polynomial $x^4+y^4+z^4$ is invariant under the octahedral group $B_3$ \cite{Meyer54}. According to $\int_{\S^2}x^4+y^4+z^4 = \frac{3}{5}$, we deduce that $ p_{\invariant}=p_{4,B_3}$ may be chosen as 
\begin{equation*}
p_{4,B_3}(x,y,z)=x^4+y^4+z^4-\tfrac{3}{5}\,,
\end{equation*}
so that it generates $\mathcal{H}_4^{B_3}=\R p_{4,B_3}$.

We apply Lemma \ref{thm:basic mixed orbit constr} and compute $\beta$ according to \eqref{eq:rho definer} for the respective roles of the cube and the octahedron as well as the cuboctahedron and the rhombic dodecahedron. 
The point evaluations of $p_{4,B_3}$ on the respective vertices $x_{cube}=\frac{1}{\sqrt{3}}(1,1,1)$, $x_{oct} = e_1$, 
and $x_{cuboct} =\frac{1}{\sqrt{2}}(e_1+e_2)$ are
\begin{equation*}
p_{4,B_3}(x_{cube}) = -\tfrac{4}{15},\qquad p_{4,B_3}(x_{oct}) = \tfrac{2}{5},\qquad p_{4,B_3}(x_{cuboct}) = -\tfrac{1}{10}\,.
\end{equation*}
Let $\gamma_{oct}$, $\gamma_{cube}$, and $\gamma_{rh\text{-}do}$ denote the geodesic arcs induced by an edge of the octahedron, cube, and rhombic dodecahedron, respectively. We may choose $\gamma_{rh\text{-}do}$ as the arc from $x_{cube}$ to $x_{oct}$. Integration yields
\begin{align*}
\frac{1}{\ell(\gamma_{oct})} \int_{\gamma_{oct}}  p_{4,B_3} & = \tfrac{3}{20} \,,\\
\frac{1}{\ell(\gamma_{cube})} \int_{\gamma_{cube}}  p_{4,B_3} & = -\tfrac{10 \sqrt{2}+3 \arccos(\frac{1}{3})}{80 \arccos \left(\frac{1}{3}\right)}\approx -0.181\,,\\
\frac{1}{\ell(\gamma_{rh\text{-}do})} \int_{\gamma_{rh\text{-}do}}  p_{4,B_3} & = \tfrac{5 \sqrt{2}-3 \arccos(\frac{1}{\sqrt{3}})}{80 \arccos (\frac{1}{\sqrt{3}})} \approx 0.055\,,
\end{align*}
so that the condition on opposite signs in Lemma \ref{thm:basic mixed orbit constr} is satisfied. The values of $\beta$ provided in Table \ref{table:Theorem} are derived by \eqref{eq:rho definer}.

\smallskip
(v-vii) Recall that $\phi = \tfrac{1+\sqrt{5}}{2}$ denotes the golden ratio. For the icosahedral group $H_3$, the following invariant polynomial 
\begin{equation*}
p_0(x,y,z)=(\phi ^2 x^2-y^2)(\phi ^2 y^2-z^2)(\phi ^2 z^2-x^2) 
\end{equation*}
is provided in \cite{Meyer54}. It is compatible with our representation of $H_3$. We may compute $\int_{\mathbb{S}^2} p_0=-\frac{2+\sqrt{5}}{21}$, so that 
$p_{6,H_3}:=p_0+\frac{2+\sqrt{5}}{21}$ generates $\mathcal{H}_6^{H_3}=\R  p_{6,H_3}$.

We evaluate $p_{4,B_3}$ at the vertices $x_{ico} =\frac{1}{\sqrt{1+\phi^2}}(\phi,1,0)$, $x_{dode}=\frac{1}{\sqrt{3}}(1,1,1)$, and $x_{icosido} = e_1$ and obtain
\begin{equation*}
p_{6,H_3}(x_{ico}) = - \tfrac{16(2+\sqrt{5})}{105}\,,\qquad p_{6,H_3}(x_{dode}) =\tfrac{16(2+\sqrt{5})}{189}\, ,\qquad p_{6,H_3}(x_{icosido}) = \tfrac{2+\sqrt{5}}{21}\,.
\end{equation*}
Let $\gamma_{dode}$, $\gamma_{ico}$, and $\gamma_{rh\text{-}tria}$ denote the geodesic arcs induced by an edge of the dodecahedron, icosahedron, and rhombic triacontahedron, respectively. We may choose $\gamma_{rh\text{-}tria}$ as an arc from $x_{ico}$ to  $x_{dode}$ and more involved symbolic calculations in Mathematica eventually yield
\begin{align*}
\frac{1}{\ell(\gamma_{dode})} \int_{\gamma_{dode}}  p_{6,H_3} & =(2+\sqrt{5}) \tfrac{238\sqrt{5}-135\arccos(\frac{\sqrt{5}}{3})}{9072\arccos(\frac{\sqrt{5}}{3})} \approx 0.277 \,,\\
\frac{1}{\ell(\gamma_{ico})} \int_{\gamma_{ico}} p_{6,H_3} & = -(2+\sqrt{5})\tfrac{ 14+5 \arccos(\frac{1}{\sqrt{5}})}{336 \arccos(\frac{1}{\sqrt{5}})}\approx -0.222\,,\\
\frac{1}{\ell(\gamma_{rh\text{-}tria})} \int_{\gamma_{rh\text{-}tria}}  p_{6,H_3} & = -(2+\sqrt{5}) \tfrac{119\sqrt{5}-189+135\arccos\left(\sqrt{\frac{1}{3}+\frac{2}{3\sqrt{5}}}\right)}{9072\arccos\left(\sqrt{\frac{1}{3}+\frac{2}{3\sqrt{5}}}\right)}\approx -0.118\,.
\end{align*}
The balancing factor $\beta$ is now determined according to \eqref{eq:rho definer} that leads to the values of $\beta$ provided in Table \ref{table:Theorem}. 
\end{proof}

\begin{remark}
The invariant polynomials $ p_{3,A_3}$, $p_{4,B_3}$, and $p_{6,H_3}$ as derived in the proof of Theorem \ref{thm:main} are shown in Figure \ref{fig:invariant}. The maximum and the minimum are attained precisely at the vertices of the associated convex regular polytopes. Moreover, the polynomials are constant on the spherical edges of the octahedron, cuboctahedron, and icosidodecahedron, respectively. 
\end{remark}

\begin{figure}
\subfigure[$A_3$: tetrahedron (blue), dual tetrahedron (red)]{
\includegraphics[width=.3\textwidth]{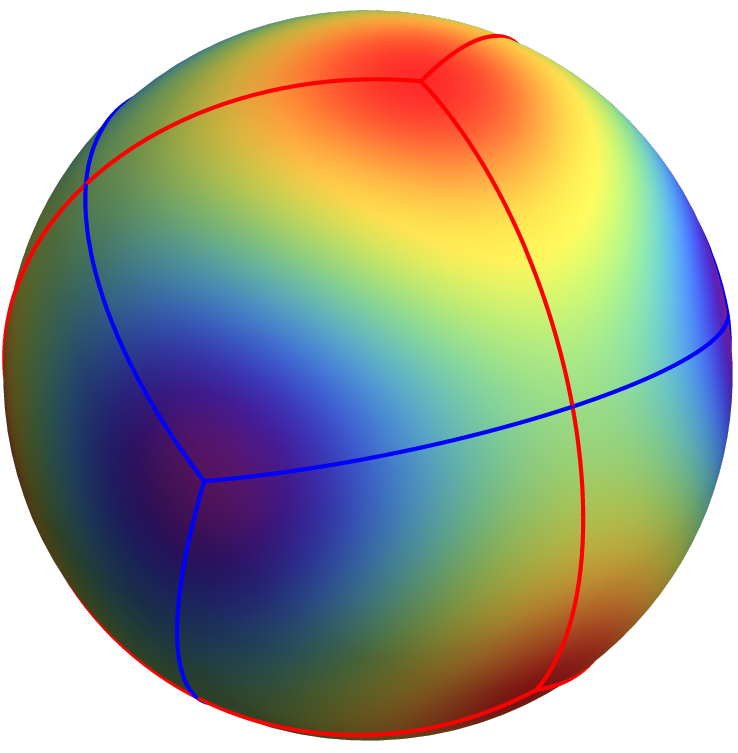}
}\hfill
\subfigure[$B_3$: cube (blue), octahedron (red)]{
\includegraphics[width=.3\textwidth]{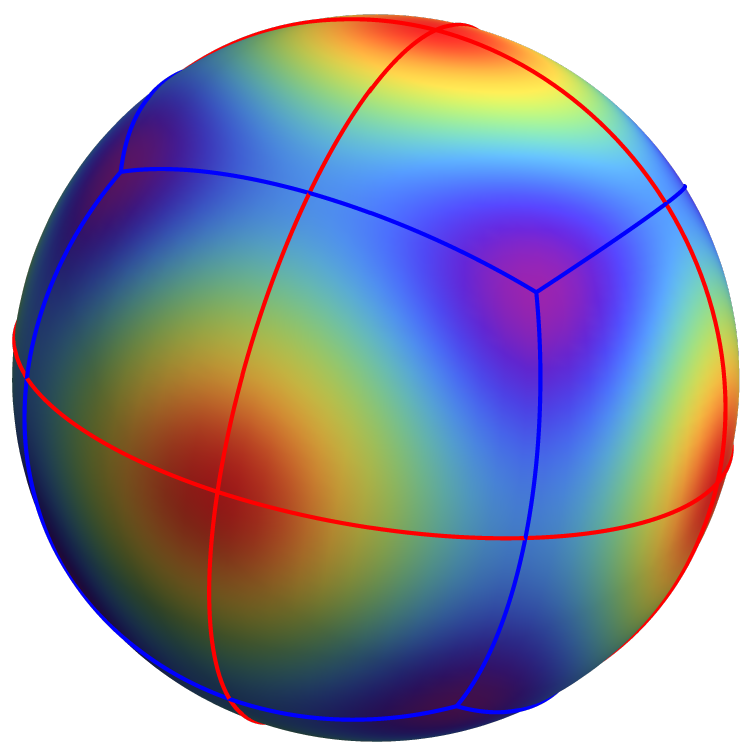}
}\hfill
\subfigure[$ H_3$: icosahedron (blue), dodecahedron (red)]{
\includegraphics[width=.3\textwidth]{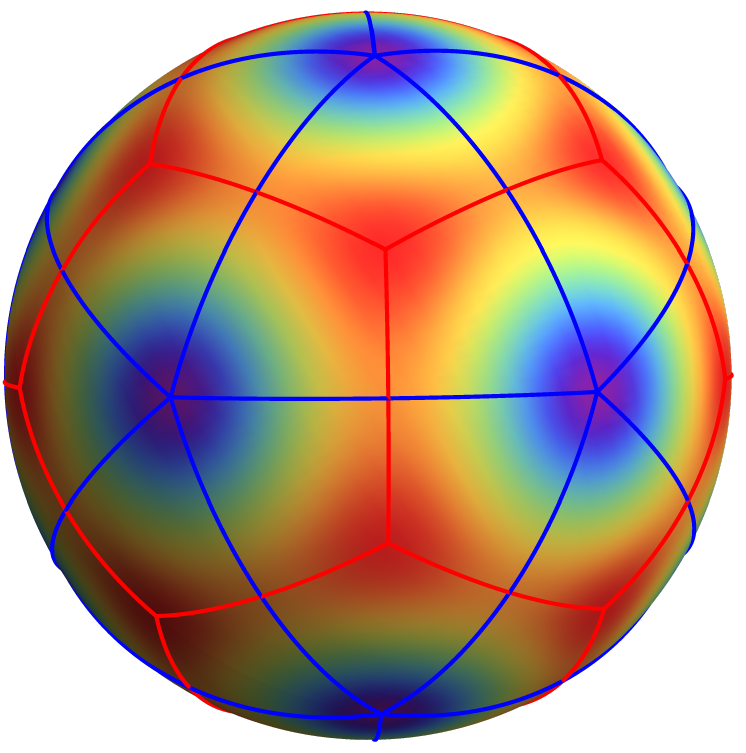}
}

\subfigure[$A_3$: octahedron (green), cube (black)]{
\includegraphics[width=.3\textwidth]{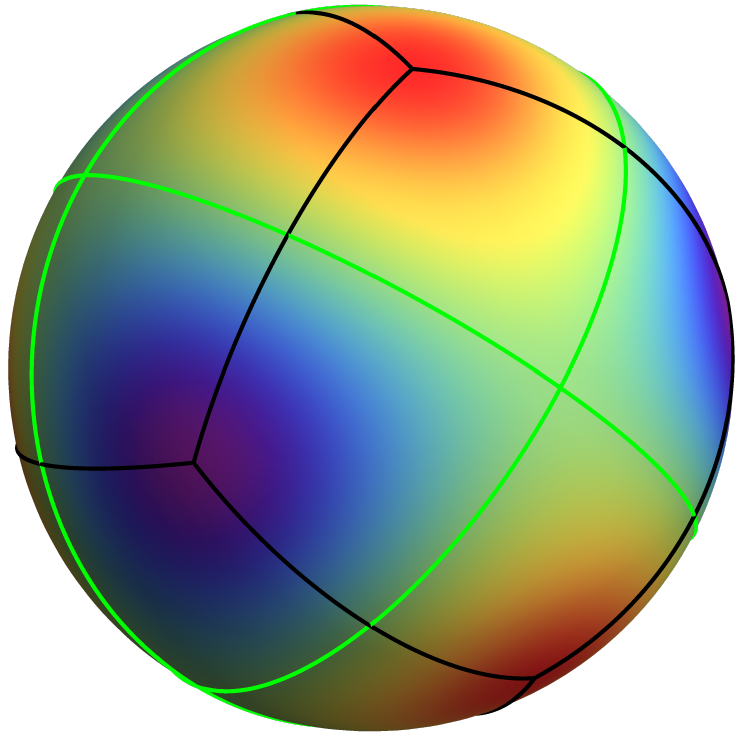}
}\hfill
\subfigure[$B_3$: cuboctahedron (green), rhombic dodecahedron (black)]{
\includegraphics[width=.3\textwidth]{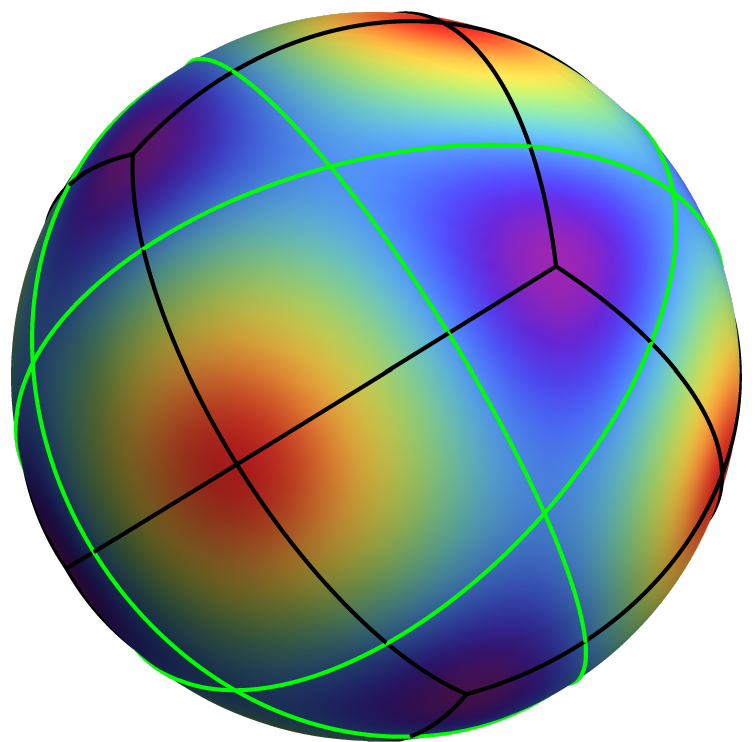}
}\hfill
\subfigure[$ H_3$: icosidodecahedron (green), rhombic triacontahedron (black)]{
\includegraphics[width=.3\textwidth]{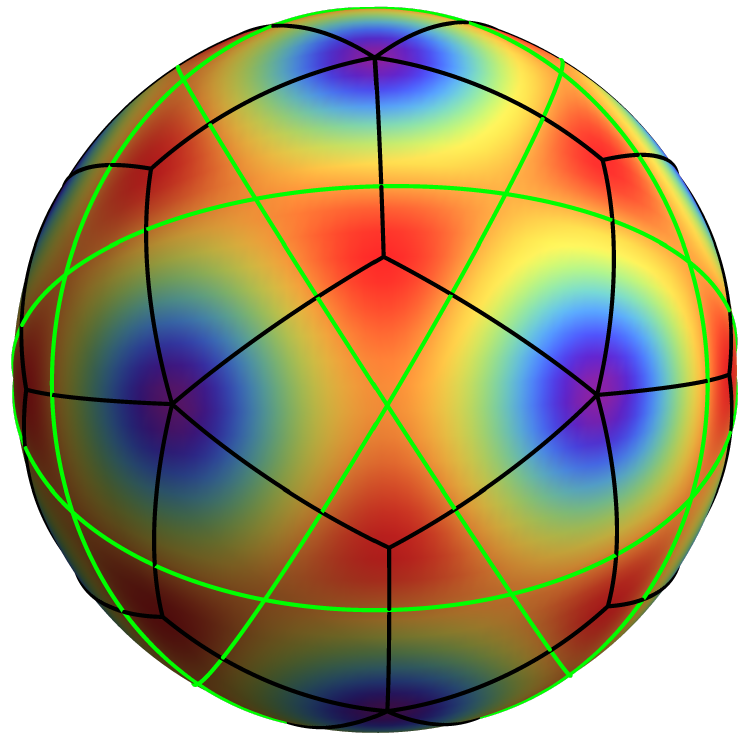}
}
\caption{The invariant polynomials $ p_{3,A_3}\in\mathcal{H}^{A_3}_3$, $p_{4,B_3}\in\mathcal{H}^{B_3}_4$, and $p_{6,H_3}\in\mathcal{H}^{H_3}_6$ are shown with the rainbow color scheme that ranges from $\min  p_{\invariant}$ (violet) to $\max  p_{\invariant}$ (red). The vertices of the associated pairs of convex regular polytopes lie precisely at the minima and maxima, respectively. The polynomials $ p_{3,A_3}$, $p_{4,B_3}$, $p_{6,H_3}$ are constant on the (green) spherical edges of the octahedron, cuboctahedron, icosidodecahedron, respectively.
}\label{fig:invariant}
\end{figure}

We conclude our discussion of hybrid designs in \(\mathbb{S}^2\) and now turn our attention to higher dimensions.

\section{Hybrid designs in higher dimensions}\label{sec:Sd}
We first derive our strongest hybrid designs that are valid in arbitrary dimensions. Then, we focus on \(\mathbb{S}^3\) and obtain the strongest hybrid design overall. 
\begin{proposition}\label{prop:all d}
The $d$-octahedron and the $d$-cube are a dual polytope pair that provides hybrid $5$-designs in $\S^{d-1}$.
\begin{itemize}
\item[\textnormal{(a)}] If the geodesic cycle is induced by the edges of the $d$-octahedron, and the vertices are taken from its dual $d$-cube, then the balancing factor is 
\begin{equation}\label{eq:oct/cube d}
\beta_{oct/cube} = \frac{3d(d-2)}{3d^2+2d-8}\,.
\end{equation}
\item[\textnormal{(b)}] For the reverse roles, the balancing factor is 
\begin{equation}\label{eq:balance dd}
\beta_{cube/oct} = 3\cdot\frac{d-2}{d+2}\cdot \frac{2(d+2)\sqrt{d-1}-d(d-4)\arccos(\tfrac{d-2}{d})}{6(d-2)\sqrt{d-1}+d(5d-8)\arccos(\tfrac{d-2}{d})}\,.
\end{equation}
\end{itemize}
\end{proposition}
\begin{proof}
The symmetry group of the $d$-cube and the $d$-octahedron is the finite Coxeter group $B_d$. According to \cite{Goethals:81b}, see also \cite[Table 1 in Section 3.7, (20) in Section 3.8]{Humphreys}, \cite{Meyer54}, its Molien series is
\begin{equation*}
M_{B_d}(\omega)=\frac{1}{(1-\omega^4)(1-\omega^6)(1-\omega^{2d})}=1+\omega^4+\omega^6+O(\omega^8)\,.
\end{equation*}
Hence, $\mathcal{H}_4^{B_d}$ is one-dimensional and the space of $B_d$-invariant polynomials of degree at most $5$ is $\Pol_{\leq 5}^{B_d} = \R\oplus \mathcal{H}_4^{B_d}$. To verify that we obtain hybrid $5$-designs, we still need to determine the balancing factor $0<\beta<1$.

The root system of $B_d$ is $\pm e_i$ and $\pm e_i\pm e_j$, for $1\leq i<j\leq d$, so that $B_d$ is generated by the reflections $\mathrm{I}_d-\frac{2}{\|v\|^2} v^\top v$, where $v$ runs through the roots. We may deduce that the polynomial $x_1^4+\ldots+x_d^4$ is $B_d$-invariant. Since $\int_{\S^{d-1}}x_1^4+\ldots+x_d^4=\frac{3}{d+2}$, the $B_d$-invariant polynomial $ p_{\invariant}=p_{4,B_d}$ that spans $\mathcal{H}_4^{B_d}=\R p_{4,B_d}$ is  
\begin{equation*}
 p_{4,B_d}(x) = x_1^4+\ldots+x_d^4-\tfrac{3}{d+2}\,.
\end{equation*}

(a) To determine $\beta_{oct/cube}$, we choose $x_0=\frac{1}{\sqrt{d}}(1,\ldots,1)$ as the vertex of the $d$-cube and observe $ p_{4,B_d}\big(\frac{1}{\sqrt{d}}(1,\ldots,1)\big) = -\tfrac{2(d-1)}{d(d+2)}$. The arc derived from the edge of the $d$-octahedron is 
$\gamma_{e_1,e_2}=\sin(\frac{\pi}{2}-s) e_1 +\sin(s)e_2$, for $s\in [0,\frac{\pi}{2}]$, and direct computations lead to 
\begin{equation*}
\frac{1}{\ell(\gamma_{e_1,e_2})} \int_{\gamma_{e_1,e_2}}  p_{4,B_d} = \frac{3(d-2)}{4(d+2)}\,.
\end{equation*}
We compute the balancing factor by \eqref{eq:rho definer} and obtain $\beta_{oct/cube}$ as in \eqref{eq:oct/cube d}.

(b) To determine $\beta_{cube/oct}$, we observe $ p_{4,B_d}(e_1)=\frac{d-1}{d+2}$. We integrate along an arc $\gamma_{u,v}$ induced by the edge of the cube from $u=\frac{1}{\sqrt{d}}(-1,1,\ldots,1)$ to $v=\frac{1}{\sqrt{d}}(1,1,\ldots,1)$ and eventually derive the balancing factor as claimed in \eqref{eq:balance dd}.
\end{proof}

The remaining section is dedicated to derive explicit hybrid designs in $\S^3$ from the six convex regular polytopes in $\R^4$.  The symmetry group of the cube-octahedron pair in $\mathbb{R}^4$ is $B_4$, it yields a hybrid $5$-design, and the balancing factors in Proposition \ref{prop:all d} simplify to 
\[
\beta_{cube/oct} = 
\frac{3\sqrt{3}}{3\sqrt{3}+4\pi}
\approx 0.29, \qquad \qquad \beta_{oct/cube} = \frac{1}{2}.
\]

The symmetries of the remaining four convex regular polytopes, the $4$-tetrahedron, the $24$-cell, as well as the $120$-cell and the $600$-cell, are the finite reflection groups $A_4$, $F_4$ and $H_4$. We derive from \cite{Goethals:81b} that their Molien series satisfy
\begin{align*}
M_{A_4}(\omega) &= \frac{1}{(1-\omega^3)(1-\omega^{4})(1-\omega^{5})}=1+\omega^3+\omega^4+O(\omega^5)\,,\\
M_{F_4}(\omega) &=\frac{1}{(1-\omega^6)(1-\omega^8)(1-\omega^{12})}=1 + \omega^6 + \omega^8 + O(\omega^{12})\,,\\
M_{H_4}(\omega) &=\frac{1}{(1-\omega^{12})(1-\omega^{20})(1-\omega^{30})}=1 + \omega^{12} + \omega^{20} + O(\omega^{24})\,.
\end{align*}
The exceptional group $H_4$ has $14400$ elements, and, individually, the $120$-cell and the $600$-cell induce pure $11$-designs as points or geodesic cycles. We realize the two polytopes as a dual pair and derive the strongest hybrid design so far.

\begin{thm}\label{thm:3 in one}
The following pairs of dual polytopes provide hybrid designs on $\S^3$.
\begin{itemize}
\item[\textnormal{(a)}] Let $\gamma$ be the geodesic cycle derived from the edges of the $600$-cell and $X$ be the vertices of the $120$-cell. Then $(\gamma,X)$ is a hybrid $19$-design with balancing factor $\beta_{600/120}=\frac{176}{301}\approx 0.58$.
\item[\textnormal{(b)}] The $24$-cell induces a hybrid $7$-design for $\beta_{24}= \frac{5}{14}\approx 0.36$.
 \item[\textnormal{(c)}] The $4$-tetrahedron yields a hybrid $3$-design with balancing factor $\beta_{tetra}=\frac{8\sqrt{15}}{8\sqrt{15} + 27\arccos(-\frac{1}{4})} \approx 0.39$.
\end{itemize}
\end{thm}
To prove this theorem, we must compute the balancing factors $\beta$ according to \eqref{eq:rho definer} and hence need explicit expressions for the respective invariant polynomials $ p_{\invariant}$. We found $p_{12,H_4}$ in \cite[Table 6]{Katsunori}, but compute the invariant polynomial also by the following alternative method.

For the three-sphere $\S^3$, the dimension of $\mathcal{H}_l$ is $\dim(\mathcal{H}_l)=(l+1)^2$, and the spherical harmonics $Y_{l,m}$ of degree $l$ provide an orthonormal basis for $\mathcal{H}_l$, so that 
\begin{equation*}
\mathcal{H}_l = \spann\{Y_{l,m}:m=1,\ldots,(l+1)^2\}\,.
\end{equation*}
The Gegenbauer polynomials $C^{(\alpha)}_l$ with $\alpha=1$, normalized by $C^{(1)}_l(1)=l+1$, coincide with the Chebychev polynomials of the second kind and satisfy the addition formula of spherical harmonics,
\begin{equation}\label{eq:add formula}
(l+1) C^{(1)}_l(\langle a,b\rangle) = \sum_{m=1}^{(l+1)^2} Y_{l,m}(a)Y_{l,m}(b)\,,\quad a,b\in\mathbb{S}^3\,.
\end{equation}
Therefore, the space $\mathcal{H}_l$ is spanned by 
$
 \spann\{x\mapsto C^{(1)}_l(\langle a,x\rangle) : a\in\mathbb{S}^3\}=\mathcal{H}_l
$. In particular, the group average of $C^{(1)}_{l}(\langle a,g \,\cdot
\rangle)$, 
\begin{equation}\label{eq:C1l}
\frac{1}{|G|} \sum_{g\in G}C^{(1)}_{l}\big(\langle a,g x
\rangle\big)\,,\qquad x\in\S^3, 
\end{equation}
is contained in $\mathcal{H}_{l}^{G}$, for all $a\in\S^3$. Hence, we may define $ p_{\invariant}$ by this average for an arbitrary $a$ provided that \eqref{eq:C1l} does not vanish completely. 

\begin{proof}[Proof of Theorem \ref{thm:3 in one}]
(a) 
The $600$-cell is edge-transitive, so that $\gamma$ is an orbit of an arc $\gamma_0$ under $G=H_4$ induced by a single edge. The invariant space $\Pol_{\leq 19}^{H_4}=\R\oplus \mathcal{H}^{H_4}_{12}$ is two-dimensional \cite{Goethals:81b}. Hence, the assumption of 
Lemma \ref{thm:basic mixed orbit constr} is satisfied and there are $x_0\in\S^3$ and $0\leq \beta<1$ such that the integration conditions \eqref{eq:int cond orbits} hold for $\Pol_{\leq 19}$, so that $(\gamma,G x_0)$ is a hybrid $19$-design.

To verify that $x_0$ can be chosen as a vertex of the $120$-cell, we must derive an explicit expression of the invariant polynomial $ p_{\invariant}=p_{12,H_4}\in \mathcal{H}^{H_4}_{12}$. To apply formula \eqref{eq:C1l}, we first generate $H_4$ as a subgroup of the orthogonal group $\mathcal{O}(4)$. 

We pick the $120$ vertices $v$ of the $600$-cell as the root system of $H_4$ given by $16$ vertices $\tfrac{1}{2}(\pm 1,\pm 1,\pm 1,\pm 1)$, $8$ vertices by all permutations of $(\pm 1,0,0,0)$, and $96$ vertices by all even permutations of $\tfrac{1}{2} (\pm \phi,\pm 1,\pm \tfrac{1}{\phi},0)$, cf.~\cite{Coxeter2}, \cite[Section 8.7]{Coxeter1}. These vectors induce $120$ reflection matrices $\mathrm{I}_4-2 v^\top v$ that we have used to generate all $14400$ matrices of $H_4$ in Mathematica. 

According to \eqref{eq:C1l}, we compute the invariant polynomial $p_{\invariant}=p_{12,H_4}$ in Mathematica by summing
\begin{equation*}
 p_{12,H_4}(x):=\frac{1}{14400} \sum_{g\in H_4}C^{(1)}_{12}\big(\langle e_1,g x
\rangle\big)\,,\qquad x\in\S^3\,.
\end{equation*}
We utilize the symmetric homogeneous polynomials $M_\lambda$ in 4 variables $x_1,x_2,x_3,x_4$, which are indexed by an integer partition $\lambda$ to express 
\begin{align*}
p_{12,H_4}(x)= M_{(12)}-22M_{(10,2)} +99M_{(8,4)} +198M_{(8,2,2)}-176M_{(6,6)}-66M_{(6,4,2)}\\
\qquad -4752M_{(6,2,2,2)} -330M_{(4,4,4)}+3960M_{(4,4,2,2)}+462\sqrt{5}\Delta_4\,,
\end{align*}
where $\Delta_4=\Pi_{1\leq i<j\leq 4}(x_i^2-x_j^2)$. 
There seems to be a sign error in \cite[Table 6]{Katsunori}, where $-462\sqrt{5}\Delta_4$ should be $462\sqrt{5}\Delta_4$. 

The arc $\gamma_0$ is induced by the edge between $(1,0,0,0)$ and $\frac{1}{2}(\phi,1,\frac{1}{\phi},0)$ of the $600$-cell, and we pick $x_0=(0,0,\frac{1}{\sqrt{2}},\frac{1}{\sqrt{2}})$ as a vertex of the $120$-cell. Remarkably, symbolic calculations deliver simple rational expressions for 
\begin{equation*}
\frac{1}{\ell(\gamma_0)}\int_{\gamma_0}  p_{12,H_4} = \frac{11}{25},\qquad\qquad  p_{12,H_4}(x_0)=-\frac{5}{16}\,,
\end{equation*}
that lead to the balancing factor $\beta_{600/120}=\frac{-\frac{11}{25}}{-\frac{5}{16}-\frac{11}{25}}=\frac{176}{301}$.

(b) The group $F_4$ has $1152$ elements and is generated by the reflections associated to $e_2-e_3$, $e_3-e_4$, $e_4$, and $(1,-1,-1,-1)$. According to \cite{Goethals:81b}, the polynomial 
\begin{equation*}
 p_{6,F_4}(x) = 16(x_1^6+x_2^6+x_3^6+x_4^6)-20(x_1^4+x_2^4+x_3^4+x_4^4)+5\,,\qquad x\in\S^3\,,
\end{equation*}
is $F_4$-invariant, and we have verified that it is orthogonal to the constants, so that $\mathcal{H}_{6}^{F_4}=\R p_{6,F_4}$. Note that $ p_{6,F_4}$ coincides with \eqref{eq:C1l} for $G=F_4$, $a=e_1$, and $l=6$.

The vertices of the $24$-cell are all permutations and sign combinations of the point $(\pm \frac{1}{\sqrt{2}},\pm \frac{1}{\sqrt{2}},0,0)$. Let $\gamma_0$ be a geodesic arc induced by an edge. The vertex of the dual realization can be chosen as $x_0=e_1$. Hence, we observe $ p_{6,F_4}(x_0)=1$, and integration yields 
\begin{equation*}
\frac{1}{\ell(\gamma_0)}\int_{\gamma_0} p_{6,F_4} =-\frac{5}{9}\,,
\end{equation*}
so that we obtain a hybrid $7$-design with balancing factor $\beta_{24} = \frac{\frac{5}{9}}{\frac{5}{9}+1}=\frac{5}{14}\approx 0.36$. 

(c) The symmetry group of the self-dual $4$-dimensional tetrahedron is $A_4$ and satisfies $\Pol_{\leq 3}^{A_4}=\R\oplus\mathcal{H}_3^{A_4}$, where $\mathcal{H}_3^{A_4}$ is one-dimensional. A representation of $A_4$ as a subgroup of $\mathcal{O}(4)$ is generated by the reflections associated to $e_1$, $e_2$, $(1,1,1,1)$, and $(0,1-\phi,\phi-2,1)$. 

For $G=A_4$, $a=(0,0,-1,1)$ and $l=3$, the expression \eqref{eq:C1l} leads to the following invariant polynomial of degree $3$,
\begin{equation*}
 p_{3,A_4}(x) = \phi\,(x_1^2x_3-x_2^2x_4)+(1-\phi)\,(x_2^2x_3-x_1^2x_4)+x_3x_4(x_3-x_4),\quad x\in\S^3\,,
\end{equation*}
so that $\mathcal{H}_3^{A_4}=\R p_{3,A_4}$.

The $5$ vertices of the primal realization of the $4$-tetrahedron is the orbit of $u_1=\frac{1}{\sqrt{2}}(0,0,1,-1)$. Another vertex is $u_2=\frac{1}{2\sqrt{2}} (0,-\sqrt{5},\phi-1,\phi)$, and let $\gamma_0=\gamma_{u_1,u_2}$ be the spherical arc between them. A vertex of the dual realization is $x_0=-u_1$. Integration and evaluation of $ p_{3,A_4}$ at $x_0$ yield
\begin{equation*}
\frac{1}{\ell(\gamma_0)}\int_{\gamma_0} p_{3,A_4} = -\frac{4\sqrt{10}}{9\sqrt{3}\arccos(-\frac{1}{4})}\,,\qquad \qquad  p_{3,A_4}(x_0) = \frac{1}{\sqrt{2}}\,.
\end{equation*}
According to \eqref{eq:rho definer}, we obtain the balancing factor $\beta_{tetra}$ as claimed.  
\end{proof}

Part (a) in Theorem \ref{thm:3 in one} yields an exceptionally strong hybrid design. Unlike most spherical \( 19 \)-designs in the literature, which are typically derived via numerical procedures on \(\mathbb{S}^2\) and \(\mathbb{S}^3\) \cite{Graf:2011lp,HardinSloane96,Womersley:2018we} and therefore lack a rigorous proof, this design is explicitly given as a group orbit. To the best of our knowledge, the only other comparable spherical \( 19 \)-design is presented in \cite{Goethals:81b} as the (point) orbit $Gx_0$, where $G=H_4$ and $x_0\in \S^3$ is a root of the invariant polynomial $p_{12,H_4}$. According to \cite{Goethals:81b}, the smallest of those orbits consists of $1440$ points. The hybrid $19$-design in Theorem \ref{thm:3 in one} for comparison is built from $720$ geodesic arcs and $600$ points.

\section{Moving beyond transitivity}\label{sec:no transitivity}
In this section, we review exact integration via invariants from a broader perspective and derive hybrid spherical designs without transitivity assumptions. It enables us to verify the hybrid designs in Figures \ref{hhh} and \ref{hh}.

Previously, we used invariants for exact integration of \(\Poltd\) by assuming \(\dim(\PoltdG) = 1\) in Lemma \ref{prop:orbit and arcs} and \(\dim(\PoltdG) = 2\) in Lemma \ref{thm:basic mixed orbit constr}. We now explain the basic idea behind using invariants to derive exact integration formulas, which holds regardless of the dimension $\dim(\PoltdG)$. This idea becomes even more evident when we explore it in a broader, more general setting. Instead of restricting ourselves to points or curves, we start with a signed measure $\mu$ on the sphere $\S^d$. The support of $\mu$ could include points, arcs, curves, or other structures.

The averaging process makes also sense for \(\mu\) by defining \(\mu^G\) as the average over the push-forwards \(g_* \mu\), i.e., 
\begin{equation}\label{eq:mu G pushforward}
\mu^G := \frac{1}{|G|}\sum_{g\in G} g_* \mu\,.
\end{equation}
This produces a \(G\)-invariant measure \(\mu^G\) and leads to the following key observation about the connection between invariant theory and exact integration. 
\begin{lemma}\label{lemma:general}
The signed measure $\mu$ integrates $\PoltdG$ exactly if and only if $\mu^G$ integrates $\Poltd$ exactly. 
\end{lemma}
\begin{proof}
The push-forward measure $g_*\mu$ satisfies the transformation formula 
\begin{equation*}
\int_{\S^d}  p(x) \,\mathrm{d}(g_*\mu)(x)=\int_{\S^d} p\circ g (x)\; \mathrm{d}\mu(x)\,,
\end{equation*}
so that the definition \eqref{eq:mu G pushforward} leads to
\begin{align*}
\int_{\S^d}  p(x) \, \mathrm{d}\mu^G(x) &= \frac{1}{|G|}\sum_{g\in G} \int_{\S^d}  p(x) \, \mathrm{d} g_*\mu(x) \\
& =\frac{1}{|G|}\sum_{g\in G} \int_{\S^d} p\circ g(x) \, \mathrm{d} \mu(x)\\
& = \int_{\S^d}  p^G(x) \,\mathrm{d} \mu(x)\,.
\end{align*}
Since the surface measure of the sphere is orthogonally invariant, we also have \(\int_{\mathbb{S}^d}  p^G = \int_{\mathbb{S}^d} p\). Hence, the equality $\int_{\S^d}  p^G(x) \, \mathrm{d}\mu(x)=\int_{\mathbb{S}^d}  p^G$ holds if and only if the identity $\int_{\S^d}  p(x) \, \mathrm{d}\mu^G(x) =\int_{\mathbb{S}^d} p$ is satisfied. 
\end{proof}

Our goal is to achieve exact integration of \(\Poltd\). Lemma \ref{lemma:general} is useful because finding a measure \(\mu\) that integrates \(\PoltdG\) exactly is easier than doing so for \(\Poltd\), as \(\PoltdG\) is a smaller space of polynomials. In the extreme case where \(\PoltdG = \mathbb{R}\), every probability measure \(\mu\) suffices. 

To make the connection to Lemma \ref{thm:basic mixed orbit constr}, take $x_0\in\S^d$ and a geodesic arc $\gamma_0$. Let $\delta_{x_0}$ denote the point measure concentrated on $x_0$ and $\nu_{\gamma_0}$ the normalized Hausdorff measure on the arc $\gamma_0$. If the probability measure $\mu$ is defined by 
\begin{equation*}
\mu = \beta\delta_{x_0} + (1-\beta)\nu_{\gamma_0}\,,
\end{equation*}
then the integration $\int_{\S^d}  p(x) \,\mathrm{d}\mu^G(x)$ coincides with the left-hand-side of \eqref{eq:int cond orbits} in Lemma \ref{thm:basic mixed orbit constr}. Thus, Lemma \ref{lemma:general} can be viewed as the general framework in which Lemmas \ref{prop:orbit and arcs} and \ref{thm:basic mixed orbit constr} operate.

Polytopes were used in the previous sections to control the support of the invariant measure $\mu^G$, so that the one-dimensional part of the support did not consist of multiple disconnected arcs but was the connected path of a continuous closed curve. 

We now eliminate the vertex- and edge-transitivity assumptions and consider points $X\subseteq \S^d$ and a geodesic cycle $\gamma\subseteq \S^d$ that are unions of orbits under a common symmetry group $G$. 
\begin{thm}\label{thm:finale}
Let $G$ be a finite subgroup of $\mathcal{O}(d+1)$ such that $\Pol_{\leq s}^G = \R\oplus\R p_{\invariant}$ is two-dimensional.  
Assume that the geodesic cycle $\gamma\subseteq\S^d$ consists of $m$ distinct orbits $G \gamma_1$, $\ldots$, $G\gamma_m$, and some finite set $X\subseteq \S^d$ consists of $n$ distinct orbits $G x_1$, $\ldots$, $Gx_n$. 
If the probability measures 
\begin{equation*}
\mu_0 = \sum_{i=1}^n \frac{|Gx_i|}{|X|} \delta_{x_i}\,,\qquad \mu_1 = \sum_{j=1}^m \frac{|G\gamma_j|\ell(\gamma_j)}{\ell(\gamma)} \nu_{\gamma_j}
\end{equation*}
satisfy the compatibility requirement that $\int p_{\invariant}\mathrm{d}\mu_0$ and $\int p_{\invariant}\mathrm{d}\mu_1$ have opposite signs, then $(\gamma,X)$ is a hybrid $s$-design.
\end{thm}
\begin{proof}
If either $\int p_{\invariant}\mathrm{d}\mu_0=0$ or $\int p_{\invariant}\mathrm{d}\mu_1=0$, we can choose $\beta=1$ or $\beta=0$, respectively, so that $\beta \int p_{\invariant}\mathrm{d}\mu_0 + (1-\beta)\int p_{\invariant}\mathrm{d}\mu_1=0$. Otherwise, we define $\beta$ analogously to \eqref{eq:rho definer} as
\begin{equation}\label{eq:beta beta}
\beta = \frac{-\int p_{\invariant}\mathrm{d}\mu_1}{\int p_{\invariant}\mathrm{d}\mu_0-\int p_{\invariant}\mathrm{d}\mu_1}\,,
\end{equation}
and also obtain $\beta \int p_{\invariant}\mathrm{d}\mu_0 + (1-\beta)\int p_{\invariant}\mathrm{d}\mu_1=0$. The opposite sign ensures that $0<\beta<1$. Therefore, $\beta \mu_0+(1-\beta)\mu_1$ integrates $\Pol_{\leq s}^G = \R\oplus\R p_{\invariant}$ exactly, and Lemma \ref{lemma:general} implies that $\beta \mu_0^G+(1-\beta)\mu_1^G$ integrates $\Pol_{\leq s}$ exactly. 

To verify that $\mu_0^G$ corresponds to $X$, we compute
\begin{align*}
\int p\, \mathrm{d}\mu^G_0 & = \frac{1}{|G|} \sum_{g\in G} \sum_{i=1}^n \frac{|Gx_i|}{|X|} p(gx_i) \\
&= \frac{1}{|X|} \sum_{i=1}^n \sum_{x\in Gx_i} p(x) =  \frac{1}{|X|} \sum_{x\in X}p(x)\,.
\end{align*}
Similarly, we verify that $\mu_1^G$ corresponds to $\gamma$ due to
\begin{align*}
\int p\, \mathrm{d}\mu^G_1 & = \frac{1}{|G|} \sum_{g\in G} \sum_{j=1}^m \frac{|G\gamma_j|\ell(\gamma_j)}{\ell(\gamma)} \frac{1}{\ell(\gamma_j)}\int_{g \gamma_j}p \\
&= \frac{1}{\ell(\gamma)} \sum_{j=1}^m \frac{|G\gamma_j|}{|G|}\sum_{g\in G}  \int_{g\gamma_j} p\\
&=\frac{1}{\ell(\gamma)} \sum_{j=1}^m  \int_{G\gamma_j} p = \frac{1}{\ell(\gamma)} \int_{\gamma} p\,.
\end{align*}
Thus, $(\gamma,X)$ is a hybrid $s$-design.
\end{proof}
Theorem \ref{thm:finale} allows the construction of hybrid spherical designs from polytopes that are neither edge- nor vertex-transitive. The rhombic dodecahedron and the rhombic triacontahedron, for instance, are not vertex-transitive and, hence, cannot serve as dual polytopes in Theorem \ref{thm:main}, but we are now able to address the non-transitive case as a consequence of Theorem \ref{thm:finale}.
\begin{corollary}\label{cor:last but not least}
Table \ref{table:Prop} presents hybrid \(s\)-designs on \(\mathbb{S}^2\), with vertices derived, among others, from the rhombic dodecahedron and the rhombic triacontahedron.
\end{corollary}

\begin{table}
\begin{tabular}{|c|c|r|}
\hline
$\mathcal{P}/\mathcal{P}'$& $s$ & $\beta$\\
\hline\hline
cuboctahedron/rhombic dodecahedron & $5$ & $ 21/25$\\
cuboctahedron/octahedron & $5$ & $1/5$\\
\hline
icosidodecahedron/rhombic triacontahedron& $9$ & $45/49$\\
icosidodecahedron/icosahedron & $9$ & $5/21$\\
\hline
\end{tabular}
\caption{Hybrid $s$-designs of Corollary \ref{cor:last but not least}. The first polytope induces the geodesic cycle through its edges and the second polytope provides the vertices.}
\label{table:Prop}
\end{table}

The first and third configurations in Table \ref{table:Prop} appear in Figures \ref{hhh} and \ref{hh} in the introduction. The second and fourth configurations are obtained by removing the points on the small spherical faces in these figures. The latter two configurations are not induced by pairs of dual polytopes.
\begin{proof}
(i) The $14$ vertices $X$ of the rhombic dodecahedron, when projected onto the sphere, correspond to two sets: the $8$ cube vertices $X_{cube}=B_3 \frac{1}{\sqrt{3}}(1,1,1)$ and the $6$ octahedron vertices $X_{oct}=B_3e_1$. For $X=X_{cube}\cup X_{oct}$, we observe $\frac{|X_{cube}|}{|X|}=\frac{8}{14}$ and $\frac{|X_{oct}|}{|X|}=\frac{6}{14}$, so that we define the measure
\begin{equation*}
\mu : =\mu_0+\mu_1= (\tfrac{8}{14}\delta_{\frac{1}{\sqrt{3}}(1,1,1)}+\tfrac{6}{14}\delta_{e_1}) + \nu_{\gamma_0}\,,
\end{equation*}
where $\gamma_0$ is the geodesic arc induced by an edge of the cuboctahedron between $\frac{1}{\sqrt{2}}(1,1,0)$ and $\frac{1}{\sqrt{2}}(1,0,-1)$. 

To determine $\beta$, we directly compute
\begin{equation*}
\tfrac{8}{14}p_{4,B_3}(\tfrac{1}{\sqrt{3}}(1,1,1))+\tfrac{6}{14}p_{4,B_3}(e_1) = \frac{2}{105}\,, \qquad\qquad  \int p_{4,B_3} \mathrm{d}\nu_{\gamma_0} = \frac{1}{10}\,.
\end{equation*}
The balancing factor $\beta$ is determined by \eqref{eq:beta beta}, and we obtain
\begin{equation*}
\beta = \frac{-1/10}{-1/10-2/105} = \frac{21}{25}\,.
\end{equation*}
Theorem \ref{thm:finale} yields that the cuboctahedron/rhombic dodecahedron pair leads to a hybrid $5$-design.

(ii) The vertices are now a single orbit and we define $\mu: =  \delta_{e_1} + \nu_{\gamma_0}$. Since $p_{4,B_3}(e_1)=\frac{2}{5}$, the defining equation \eqref{eq:beta beta} leads to $\beta =\frac{-1/10}{-1/10-2/5}=\frac{1}{5}$.

(iii) The $32$ vertices of the rhombic triacontahedron when projected onto the sphere, consist of the $12$ vertices of the icosahedron and the $20$ vertices of the dodecahedron. Analogously, we define
\begin{equation*}
\mu : =\mu_0+\mu_1=  (\tfrac{12}{32}\delta_{\frac{1}{\sqrt{1+\phi^2}}(\phi,1,0)}+\tfrac{20}{32}\delta_{\frac{1}{\sqrt{3}}(1,1,1)}) + \nu_{\gamma_0}\,,
\end{equation*}
where $\gamma_0$ is the geodesic arc corresponding to an edge of the icosidodecahedron. As above, we calculate
\begin{equation*}
\tfrac{12}{32}p_{6,H_3}(\tfrac{1}{\sqrt{1+\phi^2}}(\phi,1,0))+\tfrac{20}{32}p_{6,H_3}(\tfrac{1}{\sqrt{3}}(1,1,1))=-\tfrac{4(2+\sqrt{5})}{945}\,,\quad \int p_{6,H_3} \mathrm{d}\nu_{\gamma_0}=\tfrac{2+\sqrt{5}}{21}\,,
\end{equation*}
so that \eqref{eq:beta beta} leads to $\beta = \frac{45}{49}$. 

(iv) The vertices are again a single orbit, so that we put $\mu= \delta_{\frac{1}{\sqrt{1+\phi^2}}(\phi,1,0)} + \nu_{\gamma_0}$. We use $p_{6,H_3}(\tfrac{1}{\sqrt{1+\phi^2}}(\phi,1,0))=\frac{16(2+\sqrt{5})}{105}$ and obtain $\beta=\frac{5}{21}$.
\end{proof}

A few remarks conclude the paper.

\begin{remark}
In the proof of Theorem \ref{thm:finale}, we did not rely on \(\gamma_j\) being a geodesic arc, so that the integration property applies more generally. Theorem \ref{thm:finale} remains valid whenever the union of the distinct orbits \( G\gamma_j \), \( j = 1, \dots, m \) of continuous curves \(\gamma_1,\ldots,\gamma_m\) form a closed continuous curve.
\end{remark}

\begin{remark}
Lemma \ref{lemma:general} also fits to constructions of cubature formulas that weight distinct (point) orbits as explored in \cite{McLaren63,Salikhov75,Goethals:81b}. For instance, let $H_4$, the $120$-cell, and the $600$-cell be realized as in the proof of Theorem \ref{thm:3 in one}. Pick $x_0=(0,0,\frac{1}{\sqrt{2}},\frac{1}{\sqrt{2}})$ and $y_0=(1,0,0,0)$ as vertices of the $120$-cell and the $600$-cell, respectively. By using the invariant polynomial $p_{12,H_4}\in\mathcal{H}_{12}^{H_4}$ as derived in the proof of Part (a) of Theorem \ref{thm:3 in one}, we were able to verify that $\mu=\frac{16}{21}\delta_{x_0}+\frac{5}{21}\delta_{y_0}$ integrates $\Pol_{\leq 19}^{H_4}$ exactly. Then Lemma \ref{lemma:general} implies that 
\begin{equation*}
\mu^{H_4} = \frac{16/21}{600}\sum_{x\in X_{120}} \delta_x + \frac{5/21}{120}\sum_{x\in X_{600}}\delta_x
\end{equation*}
integrates $\Pol_{\leq 19}$ exactly, where $X_{600}$ and $X_{120}$ denote the vertex sets of the $600$-cell and the $120$-cell, respectively. This exact cubature formula has already been derived $50$ years ago in \cite{Salikhov75}.
\end{remark}

\begin{remark}
Let $V(\mathcal{H}^G_l)=\{x\in\S^d:p(x)=0 \text{ for all }p\in\mathcal{H}^G_l\}$ be the zero set of $\mathcal{H}^G_l$. We see in Figure \ref{fig:zero sets} that $V(\mathcal{H}^{A_3}_3)$, $V(\mathcal{H}^{B_3}_4)$, and $V(\mathcal{H}^{H_3}_6)$ are one-dimensional subsets of $\S^2$. 
\begin{figure}
\subfigure[$V(\mathcal{H}^{A_3}_3)$]{
\includegraphics[width=.22\textwidth]{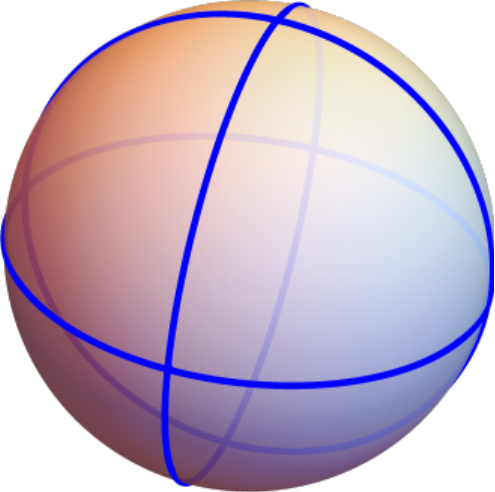}
}\hspace{1cm}
\subfigure[$V(\mathcal{H}^{B_3}_4)$]{
\includegraphics[width=.22\textwidth]{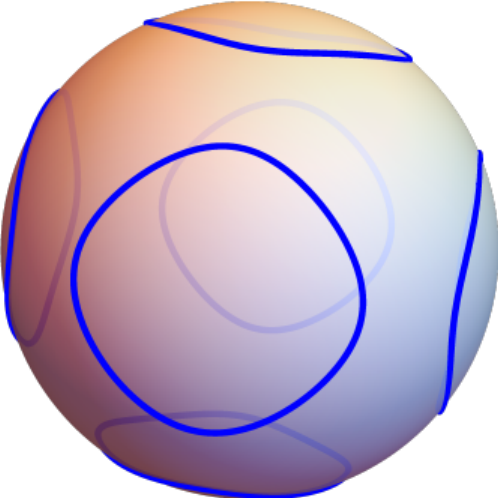}\label{fig:zero2}\label{fig:oct multiples}
}\hspace{1cm}
\subfigure[$V(\mathcal{H}^{H_3}_6)$]{
\includegraphics[width=.22\textwidth]{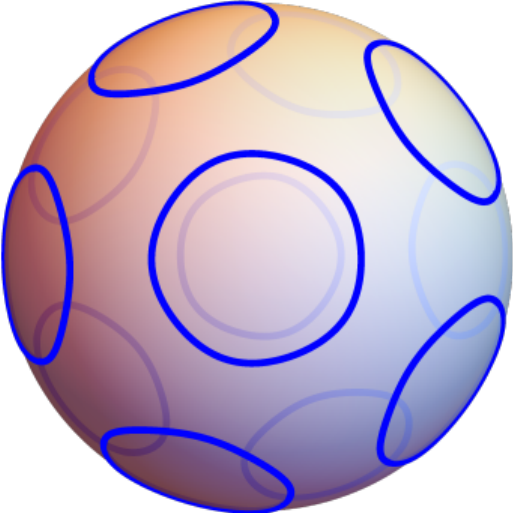}\label{fig:zero3}}
\caption{The normalized Hausdorff measures of $V(\mathcal{H}^{A_3}_3)$, $V(\mathcal{H}^{B_3}_4)$, and $V(\mathcal{H}^{H_3}_6)$ provide exact integration on $\Pol_{\leq 3}$, $\Pol_{\leq 5}$, and $\Pol_{\leq 9}$, respectively.}\label{fig:zero sets}
\end{figure}
The set $V(\mathcal{H}^{B_3}_4)$ is not connected, hence, cannot be a design curve, but consists of $6$ closed curves. Let $\mu$ be the normalized Hausdorff measure of one of them. By definition, $\mu$ integrates $\mathcal{H}^{B_3}_4=\R p_{4,B_3}$ exactly and hence also $\Pol^{B_3}_{\leq 5}=\R\oplus \mathcal{H}^{B_3}_4$. Lemma \ref{lemma:general} yields that the $B_3$-invariant measure $\mu^{B_3}$ integrates $\Pol_{\leq 5}$ exactly. Note that $\mu^{B_3}$ is simply the normalized Hausdorff measure of the $6$ closed curves. The analogous observation is made for $H_3$ with $12$ curves and $\Pol_{\leq 9}$.
\end{remark}

\section*{Acknowledgments}
The author would like to thank Karlheinz Gr\"ochenig, Martin Winter, and Clemens Karner for helpful discussions and comments.

\appendix 

\section{Proof of Lemma \ref{prop:cover}}\label{sec:covering proof}
We follow the first part of the proof of \cite[Theorem 2.2]{EG:2023}. 
According to the definition of the covering radius $\delta:=\delta(\supp(\mu))$, there is $x\in\S^d$ such that the closed ball $\mathbb{B}_{\delta/2}(x)=\{y\in\S^d:\dist(x,y)\leq \frac{\delta}{2}\}$ of radius $\delta/2$ centered at $x$ does not intersect $\supp(\mu)$, i.e., 
\begin{equation}\label{eq:00987}
\mathbb{B}_{\delta/2}(x)\cap\supp(\mu)= \emptyset.
\end{equation}
Let us denote the Laplace-Beltrami operator on the sphere $\S^d$ by $\Delta$ and the identity operator by $I$. 
According to \cite[Lemma 5.2]{Breger:2016rc},
see also \cite{Brandolini:2014oz} for the original idea, there is a constants $c_d>0$ and a
function $p$ supported on $\mathbb{B}_{\delta/2}(x)$ such that 
\begin{equation}\label{eq:123456}
\|(I-\Delta)^d p\|_{L^1(\S^d)} \leq 1,\qquad \text{ and } \qquad  \delta^{2d}\leq c_d \int_{\S^d} p.
\end{equation}
Since \eqref{eq:00987} implies $\int_{\sd}  p(x) \mathrm{d}\mu(x) = 0$, the exact integration \eqref{eq:exact integration 0} and \cite[Theorem 2.12]{Brandolini:2014oz} lead to
\begin{equation}\label{eq:654321}
\left|\int_{\S^d} p\right|
\leq c'_d t^{-2d} \|(I-\Delta)^d p \|_{L^1}\,,
\end{equation}
for some constant $c'_d>0$. 
By combining \eqref{eq:654321} with \eqref{eq:123456}, we deduce 
\begin{equation*}
\delta^{2d} \leq  c_dc_d'  t^{-2d} \|(I-\Delta)^d p\|_{L^1(\S^d)} \leq c_dc_d'  t^{-2d},
\end{equation*}
so that we conclude $\delta\leq C_d  t^{-1}$, for $C_d=(c_d c_d')^{\frac{1}{2d}}$.

\begin{landscape}
\section{Table of some edge-transitive convex polytopes}\label{sec:table}
We list some edge-transitive convex polytopes with their $t$-homogeneous symmetry groups in Table \ref{tab:list}.

{\footnotesize
\begin{table}[b]
\hfill
\begin{tabular}{c|c|c|c|c|c}
polytope & symmetry & $t$ &  $\#$ vertices &  $\#$ edges   & $\ell(\gamma)$\\
\hline
\hline
{[rectified]} $d$-tetrahedron & $A_d$ & $2$ & $d+1$ &   $\binom{d+1}{2}$,\,\, [$\,\cdot \,(d-1)$]  & $\binom{d+1}{2}\arccos(-\frac{1}{d})\frac{3-(-1)^d}{2} $\\
{[rectified]} $d$-octahedron& $B_d$ & $3$ & $2d$&  $2(d-1)d$ \,\, $[\,\cdot \,(2d-4)\,]$ & $(d-1)d \pi$\\
{[rectified]} $d$-cube & $B_d$ & $3$ & $2^d$ &  $d2^{d-1}$ \,\,[$\,\cdot \,(d-1)$] & $d\cdot 2^{d-1}\cdot\arccos(\frac{d-2}{d})\frac{3-(-1)^d}{2}$\\
\hline
{[rectified]} icosahedron & $H_3$ & $5$ & $12$ &  $30$ [$60$] & $2\cdot 30\cdot\arccos(\frac{1}{\sqrt{5}})  $\\
{[rectified]} dodecahedron & $H_3$ & $5$  & $20$ & $30$ [$60$] & $ 2\cdot 30\cdot\arccos(\frac{\sqrt{5}}{3}) $\\
\hline
{[rectified]} $24$-cell & $F_4$ & $5$ & $24$ &$96$ \,\,[$288$] &$32\pi$\\
{[rectified]} $120$-cell & $H_4$ & $11$  & $600$ &$1200$ \,\,[$3600$] & $1200\cdot \arccos(\frac{5+\sqrt{5}}{8})$\\
{[rectified]} $600$-cell & $H_4$ & $11$  & $120$ & $720$ \,\,[$3600$] & $720\cdot\arccos(\frac{1+\sqrt{5}}{4})$\\
\hline
cuboctahedron & $B_3$ & $3$ & $12$ & $24$& $8\pi$\\
rhombic dodecahedron & $B_3$ & $3$ & $14$  &  $24$ & $2\cdot 24\cdot\arccos(\frac{1}{\sqrt{3}})$\\
icosidodecahedron & $H_3$ & $5$  & $30$ & $60$ & $12 \pi$\\
rhombic triacontahedron& $H_3$ & $5$ & $32$ & $60$ & $2\cdot 60\cdot\arccos(\sqrt{\frac{5+2\sqrt{5}}{15}})$\\
\hline
$4\leq d$-demi-cube & $D_d$ & $3$  & $2^{d-1}$ & $d(d-1) 2^{d-3}$ & $2\cdot d(d-1) 2^{d-3}\cdot\arccos(\frac{d-4}{d})$ \\
$2_{21}$ & $E_6$ & $4$ & $27$ &$216$ & $216\cdot\arccos(\frac{1}{4})$\\
$3_{21}$ & $E_7$ & $5$  & $56$ & $756$ & $2\cdot 756\cdot\arccos(\frac{1}{3})$ \\
$4_{21}$ & $E_8$ & $7$  & $240$ & $6720$ & $2240 \pi$
\end{tabular}
\hfill
\phantom{T}
\caption{List of some edge-transitive polytopes with $t$-homogeneous symmetry groups. The length $\ell(\gamma)$ of the geodesic cycle $\gamma$ already contains the factor $2$ if edge doubling was required. 
}\label{tab:list}
\end{table}}
\end{landscape}


\begin{thebibliography}{10}


\bibitem{Bachoc:2002aa}
C.~Bachoc, R.~Coulangeon, and G.~Nebe.
\newblock Designs in {G}rassmannian spaces and lattices.
\newblock {\em J.~ Algebraic Combinatorics}, 16:5--19, 2002.

\bibitem{Bachoc:2004fk}
C.~Bachoc, E.~Bannai, and R.~Coulangeon.
\newblock Codes and designs in {G}rassmannian spaces.
\newblock {\em Discrete Mathematics}, 277:15--28, 2004.

\bibitem{Bajnok91}
B.~Bajnok.
\newblock Construction of designs on the $2$-sphere.
\newblock {\em Eur.~J.~Comb.}, 12(5):377--382, 1991.

\bibitem{Bannai84}
E.~Bannai.
\newblock Spherical $t$-designs that are orbits of finite groups.
\newblock {\em J.~Math.~Soc.~Japan}, 36(2):341--354, 1984.

\bibitem{Bondarenko:2011kx}
A.~Bondarenko, D.~Radchenko, and M.~Viazovska.
\newblock Optimal asymptotic bounds for spherical designs.
\newblock {\em Ann.~ Math.}, 178(2):443--452, 2013.


\bibitem{Bondarenko:2015eu}
A.~ Bondarenko, D.~ Radchenko, and M.~ Viazovska.
\newblock Well-separated spherical designs.
\newblock {\em Constr.~ Approx.}, 41(1):93--112, 2015. 


\bibitem{Danev2001}
P.~Boyvalenkov and D.~Danev.
\newblock Uniqueness of the $120$-point spherical $11$-design in four dimensions.
\newblock {\em Arch. Math.}, 77:360--368, 2001.

\bibitem{Brandolini:2014oz}
L.~Brandolini, C.~Choirat, L.~Colzani, G.~Gigante, R.~Seri, and G.~Travaglini.
\newblock Quadrature rules and distribution of points on manifolds.
\newblock {\em Annali della Scuola Normale Superiore di Pisa - Classe di
  Scienze}, XIII(4):889--923, 2014.

\bibitem{Brauchart:fk}
J.~S. Brauchart, E.~B. Saff, I.~H. Sloan, and R.~S. Womersley.
\newblock {QMC} designs: Optimal order quasi {M}onte {C}arlo integration
  schemes on the sphere.
\newblock {\em Math.~ Comp.}, 83:2821--2851, 2014.


\bibitem{Breger:2016rc}
A.~Breger, M.~Ehler, and M.~Gr\"af.
\newblock Points on manifolds with asymptotically optimal covering radius.
\newblock {\em J. Complexity}, 48:1--14, 2018.

\bibitem{Coxeter1}
H.~S.~M.~Coxeter.
\newblock Regular Polytopes.
\newblock {\em Dover Publications}, New York, 1973.

\bibitem{Coxeter2}
H.~S.~M.~Coxeter.
\newblock Regular and semi-regular polytopes. {II}
\newblock {\em Math.~Z.}, 188:559--591, 1985.

\bibitem{Harpe:2004}
P.~de~la Harpe and C.~Pache.
\newblock Spherical designs and finite group representations (some results of E. Bannai).
\newblock {\em Eur.~J.~Comb.}, 25:213--227, 2004.

\bibitem{Delsarte:1977aa}
P.~Delsarte, J.~M. Goethals, and J.~J. Seidel.
\newblock Spherical codes and designs.
\newblock {\em Geom.~Dedicata}, 6:363--388, 1977.

\bibitem{Ehler:2019aa}
M.~Ehler, M.~Gr\"af, S.~Neumayer, and G.~Steidl.
\newblock Curve based approximation of measures on manifolds by discrepancy
  minimization.
\newblock {\em Found.~Comput.~Math.}, 2021.

\bibitem{EG:2023}
M.~Ehler and K.~Gr\"ochenig. 
\newblock $t$-design curves and mobile sampling on the sphere. 
\newblock {\em Forum of Mathematics, Sigma}, 11, 2023.


\bibitem{EGK:24}
M.~Ehler, K.~Gr\"ochenig, and C.~Karner.
\newblock Geodesic cycles on the sphere: $t$-designs and Marcinkiewicz-Zygmund inequalities.
\newblock {\em arXiv:2501.06120}, 2025.

\bibitem{Etayo:2016qq}
U.~Etayo, J.~Marzo, and J.~Ortega-Cerd{\`a}.
\newblock Asymptotically optimal designs on compact algebraic manifolds.
\newblock {\em Monatsh.~ Math.}, 186(2):235--248, 2018.

\bibitem{Goethals:81b}
J.~M. Goethals and J.~J. Seidel.
\newblock Cubature formulae, polytopes, and spherical designs.
\newblock {\em  In: Davis, C., Gr\"unbaum, B., Sherk, F.A. (eds) The Geometric Vein.} Springer, New York, NY 202--218, 1981.

\bibitem{Winter23}
F.~G\"oring and M.~Winter.
\newblock The edge-transitive polytopes that are not vertex-transitive.
\newblock {\em Ars Math.~Contemp.}, 23(2), 2023.


\bibitem{Graf:2011lp}
M.~ Gr\"af and D.~ Potts.
\newblock On the computation of spherical designs by a new optimization approach based on fast spherical {F}ourier transforms.
\newblock {\em Numer.~ Math.}, 119:699--724, 2011.


\bibitem{Grochenig:2015ya}
K.~Gr{\"{o}}chenig, J.~L. Romero, J.~Unnikrishnan, and M.~Vetterli.
\newblock On minimal trajectories for mobile sampling of bandlimited fields.
\newblock {\em Appl.~ Comput.~ Harmon.~ Anal.}, 39(3):487--510, 2015.

\bibitem{GruenbaumBook}
B.~Gr\"unbaum.
\newblock Convex Polytopes.
\newblock {\em Springer Science \& Business Media}, 221, 2013.	



\bibitem{Gruenbaum2}
B.~Gr\"unbaum, G.~C.~Shephard.
\newblock Spherical tilings with transitivity properties.
\newblock {\em In: Davis, C., Gr\"unbaum, B., Sherk, F.A. (eds) The Geometric Vein.} Springer, New York, NY 65--98, 1981.

\bibitem{HardinSloane96}
R.~H.~Hardin and N.~J.~A.~Sloane.
\newblock McLaren's improved snub cube and other new spherical designs in three dimensions.
\newblock {\em Discrete Comput.~Geom.}, 15:429--441, 1996.

\bibitem{Hoggar:1982fk}
S.~G.~Hoggar.
\newblock $t$-designs in projective spaces.
\newblock {\em Eur.~ J.~ Comb.}, 3:233--254, 1982.


\bibitem{Hong82}
Y.~Hong.
\newblock On spherical $t$-designs in $\R^2$.
\newblock {\em Eur.~J.~Comb.}, 3:255--258, 1982. 

\bibitem{Humphreys}
J.~E.~Humphreys.
\newblock Reflection groups and Coxeter groups.
\newblock {\em Cambridge University Press}, 1990.

\bibitem{Katsunori}
K.~Iwasaki, A.~Kenma, K.~Matsumoto.
\newblock Polynomial invariants and harmonic functions related to exceptional regular polytopes.
\newblock {\em Exp.~Math.}, 11(2):313--319, 2002.


\bibitem{Kane}
R.~Kane.
\newblock Reflection Groups and Invariant Theory.
\newblock {\em Springer}, CMS Books in Mathematics, 2001.



\bibitem{Katznelson}
Y.~Katznelson.
\newblock An Introduction to Harmonic Analysis.
\newblock {\em Cambridge University Press}, 2004.

\bibitem{Korevaar93}
J.~Korevaar and J.~L.~H.~Meyers.
\newblock Spherical Faraday cage for the case of equal point charges and Chebychev-type quadrature on the sphere.
\newblock {\em Integral Transforms Spec.~Funct.}, 1:105--117, 1993.

\bibitem{Lindblad1}
A.~Lindblad.
\newblock Asymptotically optimal $t$-design curves on $\mathbb{S}^3$.
\newblock {\em ArXiv:2408.04044}, 2024.

\bibitem{Martini94}
H.~Martini.
\newblock A hierarchical classification of Euclidean polytopes with regularity properties. 
\newblock In {\em Polytopes: Abstract, Convex and Computational}, Springer, 71--96, 1994.

\bibitem{marzo07}
J.~Marzo.
\newblock Marcinkiewicz-{Z}ygmund inequalities and interpolation by spherical
  harmonics.
\newblock {\em J. Funct. Anal.}, 250(2):559--587, 2007.

\bibitem{marzo08}
J.~Marzo and J.~Ortega-Cerd\`a.
\newblock Equivalent norms for polynomials on the sphere.
\newblock {\em Int. Math. Res. Not. IMRN}, (5):Art. ID rnm 154, 18, 2008.

\bibitem{McLaren63}
A.~D.~McLaren.
\newblock Optimal numerical integration on a sphere.
\newblock {\em Math.~Comp.}, 17:361--383, 1963. 

\bibitem{Meyer54}
B.~Meyer.
\newblock On the symmetries of spherical harmonics.
\newblock {\em Canadian J.~Math.}, 6:135--157, 1954. 

\bibitem{Nozaki}
H.~Nozaki.
\newblock On the rigidity of spherical $t$-designs that are orbits of reflection groups $E8$ and $H4$.
\newblock {\em Eur.~J.~Comb.}, 29:1696--1703, 2008.

\bibitem{Reimer}
M.~Reimer.
\newblock Hyperinterpolation on the sphere at the minimal projection order.
\newblock {\em J.~Approx.~Theory}, 104(2):272--286, 2000.

\bibitem{Reznick:1995}
B.~Reznick.
\newblock Some constructions of spherical $5$-designs.
\newblock {\em Linear Algebra Appl.}, 226-228:163--196, 1995.

\bibitem{Seidel:2001aa}
J.~J. Seidel.
\newblock Definitions for spherical designs.
\newblock {\em J.~ Statist.~ Plann.~ Inference}, 95(1-2):307--313, 2001.

\bibitem{Salikhov75}
G.~N.~Salikhov.
\newblock Cubature formulas for a hypersphere that are invariant with respect to the group of the regular $600$-face.
\newblock {\em Dokl.~Akad.~Nauk SSSR}, 223(5):1075--1078, 1975. 


\bibitem{Sloane:2003zp}
N.~J.~A. Sloane, R.~H. Hardin, and P.~Cara.
\newblock Spherical designs in four dimensions.
\newblock {\em IEEE Information Theory Workshop}, 253--258, 2003.

\bibitem{Sobolev1962}
S.~L.~Sobolev.
\newblock Cubature formulas on a sphere invariant with respect to any finite group of rotations.
\newblock {\em Dokl. Akad. Nauk SSSR}, 146(2):310--313, 1962.

\bibitem{Sole}
P.~Sol{\'e}.
\newblock The covering radius of spherical designs.
\newblock {\em Eur.~J.~Comb.}, 12:423--431, 1991.

\bibitem{Stein:1971kx}
E.~Stein and G.~Weiss.
\newblock {\em Introduction to Fourier Analysis on Euclidean Spaces}.
\newblock Princeton University Press, Princeton, N.J., 1971.

\bibitem{Unnikrishnan:2013aa}
J.~Unnikrishnan and M.~Vetterli.
\newblock Sampling and reconstruction of spatial fields using mobile sensors.
\newblock {\em IEEE Trans.~ Signal Process.}, 61(9):2328--2340, 2013.

\bibitem{Unnikrishnan:2013df}
J.~Unnikrishnan and M.~Vetterli.
\newblock Sampling high-dimensional bandlimited fields on low-dimensional
  manifolds.
\newblock {\em IEEE Trans.~ Inform.~ Theory}, 59(4):2103--2127, 2013.

\bibitem{Wilson:1998qa}
Robin~J. Wilson.
\newblock {\em Introduction to Graph Theory}.
\newblock Addison Wesley, Longman Limited, 1998.


\bibitem{Winter00}
M.~Winter.
\newblock Eigenpolytopes, spectral polytopes and edge-transitivity.
\newblock {\em ArXiv:2009.02179}, 2020.

\bibitem{WinterPhD}
M.~Winter.
\newblock Spectral Realizations of Symmetric Graphs, Spectral Polytopes and Edge-Transitivity.
\newblock {\em PhD-thesis}, https://nbn-resolving.org/urn:nbn:de:bsz:ch1-qucosa2-752155.

\bibitem{WL21}
R.~Winter and R.~van Luijk.
\newblock The action of the Weyl group on the $E_8$ root system. 
\newblock {\em Graphs Combin.}, 37(6):1965--2064, 2021.

\bibitem{mathematica}
Wolfram Research, Inc., Mathematica, Version 14.1, Champaign, IL (2024).

\bibitem{Womersley:2018we}
R.~S. Womersley.
\newblock Efficient spherical designs with good geometric properties.
\newblock In J.~Dick, F.~Kuo, and H.~Wozniakowski, editors, {\em Contemporary
  Computational Mathematics - A Celebration of the 80th Birthday of Ian Sloan}, Springer, 1243--1285, 2018.


\end{thebibliography}
\end{document}